\documentclass[reqno,12pt]{article}

\usepackage{enumitem,stix}

\usepackage[pagebackref]{hyperref}
\usepackage[utf8]{inputenc}
\usepackage{amsmath,amsthm,stix}  
\usepackage{stmaryrd}
\usepackage{comment}
\usepackage{calligra}
\allowdisplaybreaks
\usepackage{xcolor}

\usepackage{etoolbox}

\makeatletter
\patchcmd{\upbracefill}{\m@th}{\scriptstyle\m@th}{}{}
\patchcmd{\upbracefill}{$\braceld$}{$\scriptstyle\braceld$}{}{}
\patchcmd{\upbracefill}{\bracelu}{\bracelu\mkern-1mu}{}{}
\patchcmd{\upbracefill}{\hfill\braceru}{\hfill\mkern-1mu\braceru}{}{}
\makeatother

\def\sttf2#1#2{\left[\!\!\left[#1\atop#2\right]\!\!\right]} 
\def\stss2#1#2{\left\{\!\!\left\{#1\atop#2\right\}\!\!\right\}}

\DeclareMathAlphabet{\mathcalligra}{T1}{calligra}{m}{n}
\DeclareMathAlphabet{\mathgoth}{OT1}{goth}{m}{n}

\newtheorem{theorem}{Theorem}[section]
\newtheorem{definition}[theorem]{Definition}
\newtheorem{thm}[theorem]{Theorem}
\newtheorem{prop}[theorem]{Proposition}
\newtheorem{example}[theorem]{Example}
\newtheorem{cor}[theorem]{Corollary}
\newtheorem{lemma}[theorem]{Lemma}
\newtheorem{remark}[theorem]{Remark}

\title{Combinatorics of $q,r$-analogues of Stirling numbers of type $B$}

\author{
Eli Bagno\\
\small Jerusalem College of Technology\\[-0.8ex]
\small 21 HaVaad HaLeumi St.\\[-0.8ex]
\small Jerusalem, Israel\\[-0.8ex]
\small and \\[-0.8ex]
\small Michlalah College Jerusalem\\[-0.8ex]
\small 36 Barukh Duvdevani St.,\\[-0.8ex] \small Jerusalem, Israel\\[-0.8ex]
\small
\texttt{bagnoe@g.jct.ac.il}\\\vspace{-7pt}\\ David Garber\\
\small Holon Institute of Technology\\[-0.8ex]
\small 52 Golomb St., P.O.Box 305\\[-0.8ex]
\small 5810201 Holon, Israel\\[-0.8ex]
\small \texttt{garber@hit.ac.il}
}

\date{
}

\begin{document}

\maketitle

\begin{abstract}
Stirling number of the first and the second kinds have seen many generalizations and applications in various areas of mathematics. We introduce some combinatorial parameters which realize $q$-analogues and Broder's $r$-variants of Stirling numbers of type $B$ of both kinds, which count signed set partitions and signed permutations respectively.
Applications to orthogonality relations and power sums are given. 
\end{abstract}

\section{Introduction}

Stirling numbers constitute a family of mathematical sequences that appear in various combinatorial problems. Stirling numbers of the first kind count the number of permutations of $n$ elements into $k$ disjoint cycles, while 
Stirling numbers of the second kind count the number of ways to divide a set of $n$ objects into $k$ non-empty subsets (see e.g. Stanley \cite[page 81]{EC1}).

Stirling numbers of both kinds arise in a variety of problems  in enumerative combinatorics; 
they have many combinatorial interpretations and have been generalized in various contexts and in different ways.

The (unsigned) Stirling numbers of the first kind, denoted by  $s(n,k)$, can be also defined by the following identity, which grants them the role of the entries of the transition matrix from the standard basis of the polynomial ring to its basis of falling polynomials: 
\begin{equation}\label{first kind def}
t(t+1)\cdots (t+n-1)=\sum\limits_{k=1}^n s(n,k)\cdot t^k.
\end{equation}

These numbers can be also defined via the following recursion:
$$s(n,k)=s(n-1,k-1)+(n-1)s(n-1,k).$$

\medskip

The Stirling numbers of the first kind also have a signed version $(-1)^k s(n,k)$
and the inverse of the matrix $\{ (-1)^k s(n,k)\}_{n,k}$ constitutes the Stirling numbers of the second kind, denoted by $S(n,k)$. One concludes that 
\begin{equation}\label{second kind def}
t^n=\sum\limits_{k=1}^n S(n,k)\cdot t(t-1)\cdots (t-k+1).
\end{equation}
Another way to define these numbers is via the following recursion:
$$S(n,k)=S(n-1,k-1)+k S(n-1,k).$$
 
Broder \cite{Broder} (see also \cite{d'Ocagne}) defined an $r$-version to both kinds of Stirling numbers, which counts set partitions such that the first $r$ elements are placed in $r$ distinguished parts in the case of the second kind platform, and permutations of $[n]$ which are decomposed in $k$ cycles such that the elements $1,\dots,r$ are in distinguished cycles in the case of the first kind platform. 
An excellent source that deals with this type of generalization of Stirling numbers is Mez\H{o}'s book \cite{Mezo}.

\medskip

A nice way to present set partitions is by restricted growth words, introduced by Hutchinson \cite{Hut}, and first called as such by Milne \cite{Milne}, see Definition \ref{RG-words} below. In this paper, we chose to adopt this view of Stirling numbers of the second kind. Moreover, we take this approach further to the Stirling number of the first kind. A similar presentation using the cycle decomposition of permutations in $S_n$ was contributed to us by Sagan \cite{Sagan}. 

 \medskip
 
Bala \cite{Bala} presented a generalization of the Stirling numbers of both kinds to the framework of Coxeter groups of type $B$, a.k.a. the group of signed permutations.

The algebraic definition of the Stirling numbers of the first kind, denoted by $s^B(n,k)$, as transition coefficients between bases of $\mathbb{R}_n[t]$ is as follows: 
\begin{equation}\label{generalized first kind def}
(t+1)(t+3)\cdots (t+(2n-1))=\sum\limits_{k=0}^n s^B(n,k) \cdot t^k.
\end{equation}

Combinatorially, it is known that $s^B(n,k)$ 
counts the signed permutations having $k$ non-split cycles (see Section \ref{sec signed permutation} below). 
In \cite{monthly}, we have shown a simple bijection that demonstrates this fact. 

The notion of Stirling numbers of type $B$ of the second kind, denoted by $S^B(n,k)$, is mentioned implicitly by Reiner \cite{R}, who introduced the notion of signed set partitions. 
The algebraic definition of $S^B(n,k)$ is given by the following equality: 
$$t^n= \sum\limits_{k = 0}^n S^B(n,k)\cdot (t-1)(t-3) \cdots (t-(2k-1)).$$

\medskip

This paper deals with $q$-analogues of Stirling numbers of both types in the context of signed set partitions and signed permutations, mainly from the combinatorial viewpoint. We also deal with Broder's $r$-variant of the Stirling numbers of both kinds. 

We present some new parameters which realize the $q$- and $q,r$-analogues we work with, mainly on signed permutations and on signed set partitions, and also on the ordinary  permutations. These realizations are different from the ones presented by Sagan and Swanson \cite{SaSw}.  Section 1.1 of their paper contains a summary of what is already known in this realm. 

We also introduce an additional variant of the parameter defined over signed permutations mentioned above, and we use it to derive some new results regarding orthogonality relations and sum of powers related to these $q$-variants of Stirling numbers for signed permutations and set partitions.       

\medskip

The paper is organized as follows. In Section \ref{Background}, we provide some background on restricted growth words and on the group of signed permutations, otherwise known as the Coxeter group of type $B$. The definitions of Stirling numbers of type $B$ of both kinds can be found in that section. 

Section \ref{sec q,r analogue second kind} deals with the definitions of restricted growth words for signed set partitions and of a new parameter that realizes both $q$- and $q,r$-analogues of the Stirling numbers of type $B$ of the second kind.  

In Section \ref{sec q,r analogue first kind}, we provide the definition of restricted growth words for ordinary permutations and define a new inversion number on the set of restricted growth words. We use it to obtain a realization of the $q$- as well as the $q,r$-analogues of Stirling numbers of the first kind. Next, we generalize these definitions and realizations to the framework of signed permutations. The parameter defined in this context can be considered as a flag statistic (see e.g. \cite{AR}) and thus we denote it finv. 

In Section \ref{sec additional variant finv}, we suggest a variant of the parameter finv defined in the previous section, which in some sense looks more natural. We also give a decomposition of an inversion number on signed permutations defined by Sagan and Swanson \cite{SaSw} into a sum of parameters, which can be considered as a presentation of a flag statistic. 

Section \ref{sec orthogonality} is devoted to orthogonality and power sum properties arising from Stirling numbers of type $B$ of both kinds. For proving these properties, we make use of specialization of symmetric polynomials.

\section{Background}\label{Background}

\subsection{Restricted growth words for Stirling number of type {\it A} of the second kind}

We start by recalling the definition of {\it restricted growth words} for type $A$ (see e.g. the book of E\u{g}ecio\u{g}lu and  Garsia \cite[Sec. 1.7]{EgGa}): 

\begin{definition}\label{RG-words}
The word $\omega=x_1 \cdots x_n$ over the  alphabet $\{1 , \dots , n\}$ is called a {\em restricted growth (RG)-word}  if $x_1=1$ and for each $2\leq t \leq n$ one has: $$x_t \leq \max\left\{x_1,\dots ,x_{t-1}\right\}+1.$$ 
\end{definition}
For example, $\omega=122123$ is an RG-word, but $\omega=1\mathbf{4}213$ is not. 

\medskip

There is a well-known bijection (see Milne \cite{Milne}) between restricted growth words and set partitions of the set $\{1,\dots, n \}$, which we briefly describe here. 
Let $P=\{ B_1 , \dots ,B_k\}$ be a set partition of $\{1,\dots,n\}$, whose  blocks are ordered in such a way that the set of minimum elements of the blocks is increasing.

For such a set partition, we associate an RG-word $\omega=x_1 \cdots x_n$ as follows:  
$x_j$ is the number of the block where $j$ is located. For instance, the RG-word $\omega=122123$ is matched with the set partition $\{\{\mathbf{1},4\},\{\mathbf{2},3,5\},\{\mathbf{6}\}\}$ of the set $\{1,\dots, 6\}$; note that the blocks are arranged in such a way that their minimal elements are ordered increasingly.

\subsection{Set partitions of type 
{\it B}} 
We now recall the definition of set partitions of type $B$ (see Dolgachev-Lunts \cite[p.~755]{DoLu} and Reiner \cite[Section 2]{R}; mentioned implicitly in Dowling \cite{Dow} and Zaslavsky \cite{Za} in the form of signed graphs):

\begin{definition}\label{def of type B par}
Denote: $[\pm n]:=\{\pm1,\dots,\pm n\}$. A {\em set partition of $[n]$ of type $B$} or a {\em signed set partition} is a set partition of  the set $[\pm n]$  such that the following conditions are satisfied:
\begin{itemize}
\item If $B$ appears as a block in the set partition, then $-B$ (which is obtained from $B$ by negating all its elements) also appears in that partition.
\item There exists at most one block satisfying $-B=B$. This block is called the {\em zero block} (if it  exists, it is a subset of $[\pm n]$ of the form $\{\pm i \mid i \in C\}$ for some $C \subseteq [n]$).
\end{itemize}
\end{definition}

For example, the following is a set partition of $[6]$ of type $B$:
\begin{equation*}\label{example of standard presentation}
\{\{1,-1,4,-4\},\{2,3,-5\},\{-2,-3,5\},\{6\},\{-6\}\}.    
\end{equation*}

Note that every non-zero block $B$ has a corresponding block $-B$ attached to it. For the sake of convenience, we write for the pair of blocks $B,-B$, only the representative block containing the minimal {\it positive} number appearing in $B \cup -B$. 
For example, the pair of blocks\break
$B=\{-2,-3,5\},-B=\{2,3,-5\}$ will be represented by the single block $\{\mathbf{2},3,-5\}$. 

Our convention will be to write first the zero block and denote it by $B_0$ if exists and then the non-zero blocks of a set partition of type $B$ in such a way that the sequence of absolute values of the minimal elements of the blocks is increasing. We call this the {\it standard presentation}. 

For example, the following is a set partition of $[6]$ of type $B$ in its standard presentation: 
$$\left\{B_0=\{1,-1,4,-4\},B_1=\{\mathbf{2},3,-5\},B_2=\{\mathbf{6}\}\right\}.$$

\begin{definition} \label{definition of Stirling number of type B second kind}
Let $S^B(n,k)$ be the number of set partitions of type $B$ having $k$ representative non-zero blocks. This is known as {\em the Stirling number of type $B$ of the second kind} (see sequence A085483 in OEIS \cite{OEIS}).
\end{definition}

It is easy to see that $S^B(n,n)=S^B(n,0)=1$ for each $n\geq 0$.  The following recursion for $S^B(n,k)$ is well-known (\cite[Theorem 7; see the Erratum]{Dow}, \cite[Corollary 3]{Beno}, for $m=2$, and \cite[Equation (1)]{Wa}, for $m=2,c=1$):

\begin{prop}\label{recursion theorem}
For each $1 \leq k < n$,
\begin{equation}
S^B(n,k)=S^B(n-1,k-1)+(2k+1)S^B(n-1,k).
\label{recursion_second_kind_type_b}
\end{equation}
\end{prop}

Proposition \ref{recursion theorem} is a special case of Proposition \ref{recursion q version second kind}, which will be proved in the next section.

\subsection{The group of signed permutations}\label{sec signed permutation}

The definition of the group of signed permutations is as follows:

\begin{definition}
A {\em signed permutation} is a bijective function: 
$\pi: [\pm n] \rightarrow [\pm n],$ satisfying:\break $\pi(-i)=-\pi(i)$ for all  $1 \leq i \leq n$. 
{\em The group of signed permutations of the set $[\pm n]$} (with respect to composition of functions), also known as the {\em hyperoctahedral group} or the {\em Coxeter group of type $B$}, is denoted by $B_n$. 
\end{definition}

\begin{example}\label{exam cycles}
Here is an example of a signed permutation: 
$$\pi=\left(\begin{array}{ccccc|ccccc}
-5 & -4 & -3 & -2 &-1  & 1 & 2 & 3 &  4 & 5\\
4 & -5 & 1  & -2  & 3 & -3  &2 &-1 & 5 & -4\end{array}\right) \in B_5.$$
Note that it is sufficient to know the values of $\pi(1),\dots,\pi(n)$, so that the signed permutation $\pi$ above can also be written (in window notation) as: $$\pi=[-3,2,-1,5,-4].$$ 
\end{example}

Considering $B_n$ as a subgroup of $S_{2n}$ in the natural way, we can also write every signed permutation as a product of disjoint signed cycles. 
We consider the pairs of cycles $C$ and $-C=\{-x\ | \ x \in C\} $ (if $-C \neq C$) as one unit, i.e. although $C$ and $-C$ are two disjoint cycles, we consider them as two parts of the same cycle (since they act on the same set of absolute values). For the uniqueness of the presentation, we use the convention that the minimal {\it positive} element of the pair of cycles $C$ and $-C$ appears in $C$.  

We distinguish between two types of cycles: a cycle $C$ will be called a {\it non-split cycle} (or an {\it odd cycle}) if the following condition holds: ``$i \in C$ if and only if $-i \in C$'', and will be called a {\it split cycle} (or an {\it even cycle}) otherwise. The names {\it odd} (and {\it even}, respectively) are coined according to the number of negative elements in the image of $\{1,\dots,n\}$ under that cycle. 

A signed permutation, written as a sequence of disjoint cycles, is presented in {\it  standard form} if 
its signed cycles are ordered in such a way that the sequence composed by the smallest absolute values of the elements of each cycle increases. Moreover, if $C\neq -C$, we write the cycle $C$ before the cycle $-C$.

\begin{example}\label{exam cycles1}
The signed permutation appearing in Example \ref{exam cycles} can be written as a product of disjoint cycles in standard form as follows (the smallest positive elements of each cycle are marked): 
$$\pi=({\mathbf 1},-3 ) (3,-1) ({\mathbf 2}) (-2) ({\mathbf 4}, 5,-4,-5).$$ 
Its split cycles are $(1,-3)(3,-1)$ and $(2)(-2)$, as the images of 
$\{1,2,3,4,5\}$ under the cycles $(1,-3)(3,-1)$ and $(2)(-2)$  are $\{-1,-3\}$ and $\{2\}$, respectively, so the number of negative elements in each cycle is even.
On the other hand, its non-split (or odd) cycle is $(4, 5,-4,-5)$, and the image of 
$\{1,2,3,4,5\}$ by this cycle has one negative element $-4$.
\end{example}

\begin{definition} \label{definition of Stirling number of type B first kind}
Let $s^B(n,k)$ be the number of sign permutations of type $B$ having $k$ non-split cycles. This is known as {\em the Stirling number of type $B$ of the first kind} (see sequence A132393 in OEIS \cite{OEIS}).
\end{definition}

It is easy to see that $s^B(n,n)=s^B(n,0)=1$ for each $n\geq 0$.  The following recursion for $s^B(n,k)$ is well-known, although we have not found any explicit proof:  

\begin{prop}\label{recursion theorem first kind}
For each $1 \leq k < n$,
the Stirling number of the first kind $s^B(n,k)$ satisfies the following recursion:
\begin{equation}
s^B(n,k)=s^B(n-1,k-1)+(2n-1)s^B(n-1,k).
\label{recursion_first_kind_type_b}
\end{equation}
\end{prop}

Note that the above proposition is a special case of Theorem \ref{recursion q version} in the sequel.

\section{A {\it q,r}-analogue of the Stirling number of the second kind}
\label{sec q,r analogue second kind}

This section deals with the definitions of restricted growth words for signed set partitions and of a new parameter that realizes both $q$- and $q,r$-analogues of the Stirling numbers of type $B$ of the second kind.
Some of the material appearing in this section partially overlaps our previous paper \cite{PUMA}, which contained only partial proofs. 

\subsection{The {\it r}-variant}
We present here a new generalization of the Stirling number of type $B$ of the second kind, based on the work of Broder \cite{Broder} for type $A$. The definition of the {\it $r$-Stirling number of type $B$ of the second kind} is as follows:

\begin{definition}\label{special elements}
Let $S^B(n,k,r)$ be the number of  set partitions of $[n]$ of type $B$ into $k$ non-zero blocks such that no two elements of $\{1,\dots,r\}$ are located in the same block. 
\end{definition}

For example, the following is a signed set partition counted by $S^B(7,3,2)$:
$$\{\{4,-4\},\{\mathbf{1},3,-5\},\{-1,-3,5\},\{\mathbf{2},6\},\{-2,-6\},\{\mathbf{7}\},\{-7\}\}.$$

Note that the case $r=0$ in Definition \ref{special elements} brings us back to the definition of the Stirling number of type $B$ of the second kind given in Definition \ref{definition of Stirling number of type B second kind}. 

The recursion for the $r$-Stirling numbers of type $B$ is identical to the one given in Proposition \ref{recursion theorem}, where the only differences are in the initial conditions, as the following proposition claims:

\begin{prop}\label{recursion r theorem second kind}
If $n<r$, then $S^B(n,k,r)=0$. If $n=r$, then $S^B(n,k,r)=\delta_{kr}$. If $n>r$, then: $$S^B(n,k,r)=S^B(n-1,k-1,r)+(2k+1)S^B(n-1,k,r).$$ 
\end{prop}

Proposition \ref{recursion r theorem second kind} is a special case of Proposition \ref{recursion q version second kind}, which will be proved in the next subsection.  

\subsection{Restricted growth words for set partitions of type {\it B} and their weights}

We are going to prove a $q$-analogue of Propositions \ref{recursion theorem} and \ref{recursion r theorem second kind}. To this end, we modify Definition \ref{RG-words}, to produce an extended version of restricted growth words, which will be in bijection with the set partitions of type $B$.

\begin{definition}\label{def_RGB}
Let $\Sigma_{(ii)}^B$ be the alphabet $\{0,\pm 1,\pm 2,\dots,\pm n\}$ and define the following order on $\Sigma_{(ii)}^B$:\\
\centerline{$0 \prec -1 \prec 1 \prec -2 \prec 2 \prec \cdots \prec -n \prec n. $}

\medskip

A {\em restricted growth (RG-)word of type $B$ of the second kind} is a word $\omega=\omega_1\cdots \omega_n$ in the alphabet $\Sigma_{(ii)}^B$ which satisfies the following conditions:
\begin{enumerate}
    \item[(1)] We have $\omega_1=0$ or $\omega_1=1$. 
    \item[(2)] For each $2 \leq t \leq n$, the following inequality holds: $$\omega_t\preceq \max\left\{|\omega_1|,\dots,|\omega_{t-1}|\right\}+1,$$ 
    with respect to the order defined above. \\ In the case that \begin{equation}\label{cond 2b}
    \omega_t= \max\left\{|\omega_1|,\dots,|\omega_{t-1}|\right\}+1,\end{equation}
    we demand: $\omega_t>0$.
\end{enumerate}
Denote by $R_{(ii)}^B(n,k)$ the set of RG-words of type $B$ of the second kind, such that the maximal element is $k$. 
\end{definition}

Now, let $P=\{B_0,B_1,\dots ,B_k\}$ be a set partition of $[n]$ of type $B$, written in its standard presentation (see after Definition \ref{def of type B par}). We associate $P$ with an RG-word $\omega=\omega_1 \cdots \omega_n$ of type $B$  as follows: for each $1 \leq j \leq n$, 
$\omega_j$ is the number of the representative block where $j$ or $-j$ appears.
If the element $j$ appears in the representative block, then $\omega_j$ is the number of the block containing $j$; otherwise $\omega_j$ will be the number of this block, with a negative sign.  
Note that if $j$ is the smallest element in its block (in absolute value), then it should appear in the representative block by definition, so that we demand that $\omega_j>0$ (as Equation (\ref{cond 2b}) requires).  

\begin{example}
Let 
$P=\{B_0=\{2,5,-2,-5\},\ B_1=\{\mathbf{1},-7 \},\ B_2=\{\mathbf{3}, -4,6\}\}$ be a set partition of the set $[7]$ of type $B$, given in its standard presentation. Then, its associated RG-word of type $B$ is: $\omega=(1,0,2,-2,0,2,-1).$

Note, on the other hand, that the following word:
$\omega'=(1,0,\mathbf{-2},2,0,-2,-1)$
is not an RG-word of type $B$, since the first appearance of $2$ is negative.
\end{example}

It is easy to see that this forms a bijection between 
the set of set partitions of $[n]$ of type $B$ having $k$ non-zero representative blocks and the set $R^B_{(ii)}(n,k)$.

\medskip

Motivated by Cai and Readdy \cite{CaRe}, who dealt with the Stirling number of the second kind for ordinary set partitions, we equip each RG-word of type $B$ (or, equivalently, each set partition of $[n]$ of type $B$) with a weight, which generates the $q$-Stirling numbers of type $B$ of the second kind. Note that this is the first appearance of a combinatorial interpretation of $q$-Stirling number of type $B$ of the second kind, except for the work of Sagan and Swanson \cite{SaSw}, who reached an equivalent definition independently, almost at the same time, using the explicit definition of set partitions of type $B$. 

\begin{definition}
\label{def q stirling second kind}
Let $\omega=\omega_1\cdots \omega_n\in R_{(ii)}^B(n,k)$. Define the {\em weight of $\omega$}, denoted ${\rm wt}(\omega)$, by 
$${\rm wt}(\omega)=\prod\limits_{i=1}^n {\rm wt}_i(\omega),$$ 
where ${\rm wt}_1(\omega)=1$ and for $2\leq i \leq n$, 
$${\rm wt}_i(\omega)=\left\{ \begin{array}{lll}
1 & &\omega_i>\max\{\omega_1,\dots, \omega_{i-1} \} \mbox{ {\rm or} } \omega_i=0  \\
q^{2 |\omega_i| -1}& & |\omega_i| \leq \max\{|\omega_1|,\dots, |\omega_{i-1}|\} \mbox{ {\rm and} }  \omega_i<0 \\
q^{2 |\omega_i|}& & |\omega_i| \leq \max\{|\omega_1|,\dots, |\omega_{i-1}|\} \mbox{ {\rm and} }  \omega_i>0 \\
\end{array}\right. $$

We also define the {\em $q$-Stirling number $S^B_q(n,k)$ of the second kind of type $B$} as follows:
$$S_q^B(n,k) := \sum\limits_{\omega \in R_{(ii)}^B(n,k)} {\rm wt} (\omega).$$
\end{definition}

From the definition,  it follows that for each $m \in \{\pm 1,\dots,\pm k\}$, the first occurrence of $|m|$ in $\omega$ has no contribution to the weight, but each of its next occurrences contributes $q^{2|m|-1}$ for $m<0$ and $q^{2|m|}$ for $m>0$. Moreover, the occurrences of $0$ in $\omega$ have no contribution to the weight. 

\medskip

By the bijection between the set $R_{(ii)}^B(n,k)$ and the set of set partitions of $[ n]$ of type $B$, we have:
$\left.S_q^B(n,k)\right|_{q=1}=S^B(n,k).$

\begin{example}
Given the following set partition of the set $[6]$ of type $B$, written in its standard presentation: $$\{B_0=\{2,-2\}, B_1=\{\mathbf{1},-3\},B_2=\{\mathbf{4},-5,6\}\},$$
its associated RG-word is: $\omega=(1,0,-1,2,-2,2)$, so we have: $${\rm wt} (\omega)= \underbrace{1}_{{\rm wt}_1} \cdot \underbrace{1}_{{\rm wt}_2} \cdot \underbrace{q}_{{\rm wt}_3} \cdot \underbrace{1}_{{\rm wt}_4} \cdot \underbrace{q^3}_{{\rm wt}_5} \cdot \underbrace{q^4}_{{\rm wt}_6}=q^8. $$
\end{example}

Now, we can prove combinatorially  a Stirling-type recursion for the $q$-Stirling number of type $B$ of the second kind:

\begin{prop}\label{recursion q version second kind}
For each $n \in \mathbb{N}$ and $1 \leq k< n$,
\begin{equation}
\label{recursion q second kind}
S^B_q(n,k)=S^B_q(n-1,k-1)+[2k+1]_q \cdot S^B_q(n-1,k),    
\end{equation}
with the boundary conditions: $S^B_q(n,0)=S^B_q(n,n)=1$, where $[2k+1]_q=1+q+\cdots+q^{2k}$.
\end{prop}

\begin{proof}
We start by verifying the boundary conditions. Note that $R_{(ii)}^B(n,0)$ consists of the single RG-word of type $B$ of the second kind $(0,0,\dots, 0)$ of length $n$, which corresponds to the set partition of type $B$: $\left\{B_0=\{1,-1,2,-2, \dots , n,-n\}\right\}$. 
Therefore:
$$ S^B_q(n,0)= \sum\limits_{\omega \in R_{(ii)}^B(n,0)} {\rm wt} (\omega)={\rm wt} \left((0,0,\dots,0)\right)=1.$$
Note also that $R_{(ii)}^B(n,n)$ consists of the single RG-word of type $B$ of the second kind $(1,2,\dots, n)$ of length $n$, which corresponds to the set partition of type $B$: $\left\{B_1=\{1\},B_2=\{2\},\dots, B_n=\{n\}\right\}$. 
Therefore:
$$ S^B_q(n,n)= \sum\limits_{\omega \in R_{(ii)}^B(n,n)} {\rm wt} (\omega)={\rm wt} \left((1,2,\dots n)\right)=1.$$

We now prove the recurrence relation.  Let 
$$f:R_{(ii)}^B(n,k) \rightarrow R_{(ii)}^B(n-1,k) \cupdot R_{(ii)}^B(n-1,k-1)$$ 
be the function defined by removing the last element, i.e.  $$f(\omega) =f(\omega_1\cdots\omega_n)=\omega_1\cdots\omega_{n-1}.$$     
For $\omega \in R_{(ii)}^B(n,k)$, we have exactly two possibilities: 
\begin{itemize}
\item If the maximal element $k$ of $\omega$ appears only once, at position $n$ (i.e. $\omega_n=k$ and $\omega_j<k$ for all $j<n$), then $f(\omega)\in R_{(ii)}^B(n-1,k-1)$. Actually, the RG-words satisfying this condition in $R_{(ii)}^B(n,k)$ are in bijection with $R_{(ii)}^B(n-1,k-1)$.  In this case, we have ${\rm wt}(f(\omega))={\rm wt}(\omega)$.
\item Otherwise, the maximal element $k$ appears in $\omega$ before position $n$,   i.e. $\omega_j=k$ for some $j<n$ (note that each element $1 \leq i < k$ should also appear before $\omega_n$).  Then: $f(\omega)\in R_{(ii)}^B(n-1,k)$. 
For each $f(\omega) =\omega_1 \cdots \omega_{n-1}\in R_{(ii)}^B(n-1,k),$ there are $2k+1$ possibilities for choosing the value of $\omega_n$, i.e. $\omega_n \in \{0,\pm 1,\dots, \pm k\}$. Each such possibility contributes $1,q,q^2,\dots, q^{2k}$ to the weight, respectively, giving in total $[2k+1]_q$. 
\end{itemize}

Summing up both possibilities yields the requested recursion.
\end{proof}

\subsection{The {\it q,r}-variant}
In order to establish a $q$-analogue for the $r$-variant Stirling number of type $B$ of the second kind, we define the following subset of $R_{(ii)}^B(n,k)$:

\begin{definition}\label{def_RGrB}
Let $R_{(ii)}^{B}(n,k,r)$ be the subset of $R_{(ii)}^B(n,k)$ consisting of all RG-words of type $B$ of the second kind satisfying that the first $r$ entries are $1,2,\dots,r$ in  increasing order.
\end{definition}

It is easy to see that the bijection defined above from $RG$-words of type $B$ of the second kind to set partitions of type $B$ restricts to a bijection from the set $R_{(ii)}^{B}(n,k,r)$ to the set partitions of $[n]$ of type $B$ having $k$ non-zero blocks, such that no two elements of the set $\{1,2,\dots, r\}$ are in the same non-zero block, and none of them appear in the zero block. Therefore, we define the
{\em $q,r$-Stirling number $S^B_q(n,k,r)$ of type $B$ of the second kind} as follows:
$$S^B_q(n,k,r) := \sum\limits_{\omega \in R_{(ii)}^{B}(n,k,r)} {\rm wt} (\omega).$$

Then, we have the following recurrence relation:
\begin{prop}\label{recursion q r version second kind}
For each $r \leq k< n$,
$$S^B_q(n,k,r)=S^B_q(n-1,k-1,r)+[2k+1]_q \cdot S^B_q(n-1,k,r),$$
with the boundary conditions:
$$S^B_q(r,r,r)=1 \qquad \mbox{ and } \qquad S^B_q(n,k,r)=0 \mbox{\ \  for \ \ } k<r,\ \  k>n \mbox{ or } n<r.$$
\end{prop}

\begin{proof}
The proof of the recurrence is identical to the proof of the parallel one in Proposition \ref{recursion q version second kind}, so we check here only the boundary conditions.
Note that $R_{(ii)}^{B}(r,r,r)$ consists of the single RG-word of type $B$ of the second kind $\omega=(1,2,\dots, r)$, which corresponds to the set partition of the set $[r]$ of type $B$: $\left\{B_1=\{1\},B_2=\{2\}, \dots , B_r=\{r\}\right\}$. 
Therefore:
$$S^B_q(r,r,r)= \sum\limits_{\omega \in R_{(ii)}^{B}(r,r,r)} {\rm wt} (\omega)={\rm wt} ((1,2,\dots,r))=1.$$
If $k<r$,\ \ \  $k>n$ or $n<r$, we have by definition $R_{(ii)}^{B}(n,k,r)=\emptyset$, and therefore the sum vanishes.  
\end{proof}

\section{Restricted growth words and {\it q,r}-Stirling numbers of the first kind}
\label{sec q,r analogue first kind}

In this section, we introduce an extension of the concept of restricted growth words, in order to count permutations according to their number of cycles, enumerated by the ordinary Stirling number of the first kind (Section \ref{sec type A first kind}). In section \ref{sec type B first kind}, we generalize this definition to 
the framework of signed permutations and the Stirling number of type $B$ of the first kind.

\subsection{Type {\it A}}\label{sec type A first kind}

\subsubsection{Restricted growth words for permutations}

We now define a version of {\it restricted growth words} representing permutations in $S_n$, using their cycle structure, introduced to us by Bruce Sagan \cite{Sagan}:  

\begin{definition}\label{definition of RGA}
Let $\Sigma_{(i)}^A$ be the alphabet $\{(i,j) \mid 1 \leq i,j \leq n\}$. A {\em restricted growth 
(RG-)word of type $A$ of the first kind} is a word $\omega=\omega_1\cdots \omega_n=(i_1,j_1) \cdots (i_n,j_n)$ in the alphabet $\Sigma^A_{(i)}$, which satisfies the following conditions:
\begin{enumerate}
    \item[(1)] We have $(i_1,j_1)=(1,1)$. 
    \item[(2)] For each $2 \leq t \leq n$, the following inequality holds: 
    \begin{equation}
    i_t \leq \max\left\{i_1,\dots,i_{t-1}\right\}+1,\end{equation} 
    \item[(3a)] If $i_t=\max\left\{i_1,\dots,i_{t-1}\right\}+1$,
    we have: $j_t=1$. 
    \item[(3b)] If $i_t \leq \max\left\{i_1,\dots,i_{t-1}\right\}$, then the element $(i_t, j_t-1)$ must exist in $\omega$ (not necessarily before the element $(i_t, j_t)$). 
    \end{enumerate}
We denote by $R_{(i)}^A(n,k)$ the set of all RG-words of type $A$ of the first kind satisfying $$\max\left\{i_1,\dots, i_n\right\}=k.$$

\end{definition}

We define a mapping $\Phi^A$ between  permutations in $S_n$, written as a product of cycles in standard form (i.e. the minimal element in each cycle precedes the other elements in that cycle and the cycles are ordered such that their minimal elements form an increasing sequence) and restricted growth words of type $A$ of the first kind as follows:
Let $\pi=C_1 \cdots C_k$ be written as a product of cycles in standard form, then $\Phi^A(\pi)=\omega_1\cdots \omega_n$ is defined as follows: 
$$\omega_i=(t\ , \ {\rm loc}(i,C_t)),$$
where ${\rm loc}(i,C_t)$ is the location of $i$ in the cycle $C_t$.

\begin{example}
The permutation (presented in standard form)
$$\underbrace{(\mathbf{1},7)}_{C_1} \underbrace{(\mathbf{2},5,4,9)}_{C_2}\underbrace{(\mathbf{3},8)}_{C_3}\underbrace{(\mathbf{6})}_{C_4} \in S_9$$
corresponds to the RG-word of type $A$ of the first kind: $$\omega=\underbrace{(1,1)}_{\omega_1}\underbrace{(2,1)}_{\omega_2}\underbrace{(3,1)}_{\omega_3}\underbrace{(2,3)}_{\omega_4}\underbrace{(2,2)}_{\omega_5}\underbrace{(4,1)}_{\omega_6}\underbrace{(1,2)}_{\omega_7}\underbrace{(3,2)}_{\omega_8}\underbrace{(2,4)}_{\omega_9}.$$
\end{example}

We describe now how to recover a permutation from an RG-word of type $A$ of the first kind. Given a word $\omega=\omega_1 \cdots \omega_n \in R_{(i)}^A(n,k)$, we construct a permutation $\pi$ via its cycle structure as follows: for each $t$ satisfying $\omega_t=(i_t,1)$, we open a cycle starting with the element $t$. For each $\omega_t=(i_t,j_t)$ with $j_t\neq 1$, we put the element $t$ in the cycle numbered $i_t$ in the $j_t\,^{\rm th}$ \ place. 

\subsubsection{A {\it q}-analogue}\label{sec q-analogue type A}

A natural way of defining a $q$-analogue of the Stirling number of the first kind is via the following recursion:
\begin{equation}\label{recurrence A}
s^A_q(n,k)=s^A_q(n-1,k-1)+[n-1]_q \cdot s^A_q(n-1,k),
\end{equation}
with the boundary condition: $s^A_q(0,k)=\delta_{0k}$.
We present now a combinatorial realization for this $q$-analogue. 
In order to do that, we first define a linear order $\preceq_{\rm lex}$ on the alphabet $\Sigma_{(i)}^A$ to be the usual lexicographic order:
$$(i,j) \preceq_{\rm lex} (i',j') \qquad  \Longleftrightarrow \qquad \left(i < i'\right) \ \mbox{  {\rm or}  } \ \left(i=i' \ \mbox{ {\rm and}  } \ j \leq j' \right),$$
with the convention that $\omega_i \prec_{\rm lex} \omega_j$ means that $\omega_i \preceq_{\rm lex} \omega_j$ and $\omega_i \neq \omega_j$.

\medskip

The following is the definition of an inversion statistic over the set of RG-words of the first kind: \begin{definition}
Let $\omega=\omega_1\cdots \omega_n\in R_{(i)}^A(n,k)$. Define the {\em inversion of $\omega$} by $${\rm inv}_A(\omega)=\#\left\{(\omega_i,\omega_j) \mid i<j,\ \omega_j \prec_{\rm lex} \omega_i \right\}.$$ 
\end{definition}

\begin{example}\label{example of RG1A word}
The permutation $$\pi=\underbrace{(1,7)}_{C_1} \underbrace{(2,5,4,9)}_{C_2}\underbrace{(3,8)}_{C_3}\underbrace{(6)}_{C_4},$$
which corresponds to the RG-word: $$\omega=\underbrace{(1,1)}_{\omega_1}\underbrace{(2,1)}_{\omega_2}\underbrace{(3,1)}_{\omega_3}\underbrace{(2,3)}_{\omega_4}\underbrace{(2,2)}_{\omega_5}\underbrace{(4,1)}_{\omega_6}\underbrace{(1,2)}_{\omega_7}\underbrace{(3,2)}_{\omega_8}\underbrace{(2,4)}_{\omega_9},$$
satisfies the following:
$${\rm inv}_A (\omega)= \#\left\{
\begin{array}{c}
(\omega_2,\omega_7), (\omega_3,\omega_4),(\omega_3,\omega_5),(\omega_3,\omega_7),(\omega_3,\omega_9),
(\omega_4,\omega_5), \\
(\omega_4,\omega_7),
(\omega_5,\omega_7),(\omega_6,\omega_7),
(\omega_6,\omega_8),(\omega_6,\omega_9),
(\omega_8,\omega_9)    
\end{array}\right\}=12. $$
\end{example}



\medskip

We use the parameter ${\rm inv}_A$ to provide a combinatorial realization for this $q$-analogue of the Stirling number of type $A$ defined above:

\begin{prop}\label{recursion A q version}
The generating function of the statistic ${\rm inv}_A$ satisfies: 
$$ \sum\limits_{\omega \in R_{(i)}^A(n,k)} q^{{\rm inv}_A (\omega)}\ \  = \ \  s_q^A(n,k). $$
Hence, the parameter  ${\rm inv}_A$ is a combinatorial realization of the $q$-Stirling number of type $A$ of the first kind. 
\end{prop}

\begin{proof}
We have to prove that the generating function  $\sum\limits_{\omega \in R_{(i)}^A(n,k)} q^{{\rm inv}_A (\omega)}$ satisfies the recurrence (\ref{recurrence A}) and the boundary conditions: $\sum\limits_{\omega \in R_{(i)}^A(0,k)} q^{{\rm inv}_A (\omega)}=\delta_{0k}$.

We start with proving the recurrence. Let 
$$f:R_{(i)}^A(n,k) \rightarrow R_{(i)}^A(n-1,k) \cupdot R_{(i)}^A(n-1,k-1)$$ 
be the function defined by removing the last element, i.e.  $$f(\omega_1\cdots\omega_n)=\omega_1\cdots\omega_{n-1}.$$     
For $\omega =\omega_1\cdots\omega_n =(i_1,j_1) \cdots (i_n,j_n) \in R_{(i)}^A(n,k)$, we have exactly two possibilities: 
\begin{itemize}
\item If $i_n=\max\{i_1,\dots,i_{n-1}\}+1$, and hence $j_n=1$ (i.e. $n$ is located alone in a separated cycle), then $f(\omega)\in R_{(i)}^A(n-1,k-1)$. Actually, the RG-words satisfying this condition in $R_{(i)}^A(n,k)$ are in bijection with $R_{(i)}^A(n-1,k-1)$. In this case, we have 
${\rm inv}_A(f(\omega))={\rm inv}_A(\omega),$
since $(i_t,j_t) \preceq_{\rm lex} (i_n,1)$ for $1 \leq t <n$.

\item Otherwise, we have $i_n \leq \max\{i_1,\dots,i_{n-1}\}$
(i.e. $n$ is not in a separated cycle). Then\break $f(\omega)\in R_{(i)}^A(n-1,k)$. Note that 
$i_n$ can create up to $n-2$ new inversions, depending on the number of elements $\omega_t=(i_t,j_t)$ (for $t<n$) satisfying $(i_n,j_n)\prec_{\rm lex}(i_t,j_t)$ (note that $(i_1,j_1)=(1,1)\prec_{\rm lex}(i_n,j_n)$ and therefore an addition of $n-1$ inversions is impossible). 
Hence, we have: $${\rm inv}_A(\omega)={\rm inv}_A(f(\omega))+\ell,$$
where $\ell \in \{0,1,\dots,n-2\}$ is the number of inversions created by the addition of the element $\omega_n= (i_n,j_n)$.
Note that in the case that $(i_n,j_n)$ does create an inversion, the elements of the form $(i_n,j_k)$ for $j_k \geq j_n$ will be modified to $(i_n,j_k+1)$, respectively (this change creates no new inversions with respect to the elements $\omega_t$ where $t<n$); see Example \ref{exam-push} below. 
\end{itemize}

Summing up both possibilities yields the requested recurrence.

\medskip

The boundary conditions $s^A_q(0,0)=1$ and $s^A_q(0,k)=0$ for $k \neq0$ are given by convention. 
\end{proof}

\begin{example}\label{exam-push}
We return to Example \ref{example of RG1A word} and show the influence of inserting the element\break $(i_{10},j_{10})=(2,3)$ (corresponding to inserting the number $10$ into the second cycle in the third place of $\pi=(1,7)(2,5,4,9)(3,8)(6)$, after the number $5$). 
The RG-word obtained is (the changes corresponding to the insertion of $(i_{10},j_{10})=(2,3)$ are marked in bold): 
$$w=\underbrace{(1,1)}_{\omega_1}\underbrace{(2,1)}_{\omega_2}\underbrace{(3,1)}_{\omega_3}\underbrace{{\bf(2,4)}}_{\omega_4}\underbrace{(2,2)}_{\omega_5}\underbrace{(4,1)}_{\omega_6}\underbrace{(1,2)}_{\omega_7}\underbrace{(3,2)}_{\omega_8}\underbrace{{\bf (2,5)}}_{\omega_9}\underbrace{(2,3)}_{\omega_{10}}.$$
Note that the addition of the last pair $(2,3)$ creates five new inversions:
$$((3,1),(2,3)),\ \ ((2,4),(2,3)),\ \ ((4,1),(2,3)),\ \ ((3,2),(2,3)),\ \ ((2,5),(2,3)),$$
while the change of the pairs $(2,3)$ and $(2,4)$ to $(2,4)$ and $(2,5)$, respectively, does not create any additional inversions, as promised.
\end{example}

\begin{remark}
Note that Proposition \ref{recursion A q version} appears in a different language as Exercise 22(b) in Chapter 3 of Sagan's book \cite[p. 113]{SaganBook}. Note though that the statistic {\rm maj} appearing there should be replaced by the statistic {\rm inv} \cite{Sagan personal}.
\end{remark}

\subsubsection{A {\it q,r}-analogue} \label{sec q,r-analogue type A}

We first define the {\it $r$-variant of restricted growth words} of the first kind: 
\begin{definition}\label{def_RG_kind1_ver_r}
An {\em $r$-restricted growth (RG-)word of the first kind} is a restricted growth word of the first kind $\omega=\omega_1\cdots \omega_n=(i_1,j_1)\cdots (i_n,j_n)$ in the alphabet $\Sigma^A_{(i)}$, satisfying the following additional  condition: the first $r$ pairs of $\omega$ are:
$(1,1)(2,1) \cdots (r,1).$

We denote by $R^A_{(i)}(n,k,r)$ the set of all $r$-RG-words of the first kind satisfying $\max i_t=k.$
\end{definition}

In terms of the cycle decomposition of ordinary permutations, the additional condition is translated to the requirement that no two elements of   $1,2,\dots,r$ appear in the same cycle.

\medskip

Now we introduce the $r$-variant of the $q$-analogue of the Stirling number of the first kind:

\begin{definition}
The {\em $q,r$-Stirling number $s_{q}^A(n,k,r)$ of the first kind} is defined as follows:
$$s_{q}^A(n,k,r) := \sum\limits_{\omega \in R_{(i)}^{A}(n,k,r)} q^{{\rm inv}_A (\omega)}.$$
\end{definition}

As above, we use the parameter ${\rm inv}_A$ to provide a combinatorial realization for this $q,r$-analogue of the $r$-Stirling number of type $A$:

\begin{prop}\label{recursion q,r version type A}
For each $1 \leq r < k\leq n$,
\begin{equation}
s^A_q(n,k,r)=s^A_q(n-1,k-1,r)+[n-1]_q \cdot s^A_q(n-1,k,r),
\label{recursion q,r stirling A}
\end{equation}
with the boundary conditions: $s_{q}^A(r,r,r)=1$, 
$s_{q}^A(n,k,r)=0$\ \ for \ $0\leq k < r$, and $s^A_{q}(0,k,r)= \delta_{0k}$.
\end{prop}

\begin{proof}
We proceed directly to the boundary conditions, since the recursion is identical to the one given in Proposition \ref{recursion A q version}. 

Regarding the first boundary condition, the polynomial $s^A_{q}(r,r,r)$ is the generating function of the $r$-RG-word of the first kind $\omega=\omega_1\cdots \omega_r=(1,1) \cdots (r,1)$  contributing a weight $q^{0}=1$. 

The last two boundary conditions $s_{q}^A(n,k,r)=0$ for $0\leq k < r$ and $s^A_q(0,k,r)= \delta_{0k}$ are conventions. 
\end{proof}

\subsection{Type {\it B}}\label{sec type B first kind}

\subsubsection{Restricted growth words}

For introducing a $q$-analogue for the Stirling number of type $B$ of the first kind, we define {\it restricted growth words} for this context: 

\begin{definition}\label{def_RGB_kind1}
Let $\Sigma^B_{(i)}$ be the alphabet $\{(i,j)\in \mathbb{Z}\times \mathbb{Z} \mid 1 \leq |i|,|j| \leq n\}$. A {\em restricted growth (RG-)word of type $B$ of the first kind} is a word 
$$\omega=\omega_1\cdots \omega_n=(i_1,j_1)\cdots (i_n,j_n)$$ in the alphabet $\Sigma^B_{(i)}$, which satisfies the following conditions:
\begin{enumerate}
    \item[(1)] We have either  $(i_1,j_1)=(1,1)$ or $(i_1,j_1)=(-1,1)$. 
    \item[(2)] For each $2 \leq t \leq n$, the following inequality holds: 
    \begin{equation*}
     |i_t| \leq \max\left\{|i_1|,\dots,|i_{t-1}|\right\}+1.\end{equation*} 
    \item[(3a)] If $|i_t|=\max\left\{|i_1|,\dots,|i_{t-1}|\right\}+1$,
    we have: $j_t=1$. 
    \item[(3b)] If $|i_t| \leq \max\left\{|i_1|,\dots,|i_{t-1}|\right\}$,
    one of the following pairs exists in $\omega$: either $(i_t, j_t-1)$ or $(i_t,-(j_t-1))$ (not necessarily before the element $(i_t, j_t)$).
\end{enumerate}
We denote by $R^B_{(i)}(n,k)$ the set of all RG-words of type $B$ of the first kind satisfying $$\#\{i_t\mid 1\leq t\leq n,\ i_t<0\}=k.$$
\end{definition}

\medskip

We define a mapping $\Phi^B$ between signed permutations in $B_n$ and restricted growth words of type $B$ of the first kind as follows:
Let $\pi=C_1 \cdots C_k$ where for each $1 \leq \ell \leq k$, $C_\ell$ is a signed cycle in standard form. Let $\Phi^B(\pi)=\omega_1\cdots \omega_n$ be defined according to the following rule, for each $1 \leq i \leq n$: 
$$\omega_i=\left((-1)^{{\rm par}(C_t)} \cdot t\ , \ {\rm sign}(i,C_t) \cdot {\rm loc}(i,C_t)\right),$$
where:
\begin{itemize}
\item $i$ (or $-i$) appears in the cycle $C_t$,
\item ${\rm par}(C_t)$ is $0$ if $C_t$ is a split cycle, and $1$ otherwise,
\item ${\rm sign}(i,C_t)$ is the sign of the first appearance of $i$ or $-i$ in $C_t$, and
\item ${\rm loc}(i,C_t)$ is the location of the first appearance of $i$ or $-i$ in the cycle $C_t$.
\end{itemize}

\medskip

\begin{example}
The signed permutation written in standard form (where the non-split cycles are emphasized)
$$\underbrace{(1,-7)}_{C_1}\underbrace{(-1,7)}_{-C_1} \underbrace{(2,-5,4,-9)}_{C_2}\underbrace{(-2,5,-4,9)}_{-C_2}\underbrace{\mathbf{(3,8,-3,-8)}}_{C_3=-C_3}\underbrace{\mathbf{(6,-6)}}_{C_4=-C_4},$$
corresponds to the RG-word of type $B$ of the first kind: $$\omega=\underbrace{(1,1)}_{\omega_1}\underbrace{(2,1)}_{\omega_2}\underbrace{(-3,1)}_{\omega_3}\underbrace{(2,3)}_{\omega_4}\underbrace{(2,-2)}_{\omega_5}\underbrace{(-4,1)}_{\omega_6}\underbrace{(1,-2)}_{\omega_7}\underbrace{(-3,2)}_{\omega_8}\underbrace{(2,-4)}_{\omega_9}.$$
\end{example}

We describe now the other way around, i.e. how to obtain a signed permutation in standard form  from a restricted growth word of type $B$ of the first kind. Given a word $\omega=\omega_1 \cdots \omega_n \in R^B_{(i)}(n,k)$, we construct a signed permutation $\pi$ via its cycle structure as follows: for each $\omega_t=(i_t,j_t)$ with $j_t=1$, we open a cycle starting with the element $|i_t|$, which will be a split cycle if and only if $i_t>0$. For each $\omega_t=(i_t,j_t)$ with $j_t\neq 1$, we insert ${\rm sign}(j_t)\cdot t$ to the cycle numbered $|i_t|$ in the $|j_t|^{\rm  th}$ place. For each of the cycles, we complete its negative counterpart according to whether it is split or non-split.

\medskip

Next, we define the {\it $r$-variant of restricted growth words} of type $B$ of the first kind: 
\begin{definition}\label{def_RGB_kind1_ver_r}
An {\em $r$-restricted growth (RG-)word of type $B$ of the first kind} is a restricted growth (RG-)word of type $B$ of the first kind $$\omega=\omega_1\cdots \omega_n=(i_1,j_1)\cdots (i_n,j_n)$$ in the alphabet $\Sigma^B_{(i)}$, satisfying the following additional  condition: the first $r$ pairs of $\omega$ are:
$$(-1,1)(-2,1) \cdots (-r,1).$$

We denote by $R^B_{(i)}(n,k,r)$ the set of all $r$-RG-words of type $B$ of the first kind satisfying $$\#\{i_t\mid 1\leq t\leq n,\ i_t<0\}=k.$$
\end{definition}
In terms of the cycle decomposition of signed permutations, the additional condition is translated to the requirement that the elements  $1,2,\dots,r$ appear in non-split cycles, but no two of them  share a common cycle.

\subsubsection{The {\it q}-Stirling numbers of type {\it B}  of the first kind}

We now define a $q$-analogue for the Stirling numbers  of type $B$ of the first kind, similarly to what we have done in type $A$ (see Section \ref{sec q-analogue type A} above). We start by defining a linear order $\preceq_{\rm abslex}$ on the elements on a given RG-word of type $B$ of the first kind as follows:
$$(i,j) \preceq_{\rm abslex} (i',j') \qquad  \Longleftrightarrow \qquad \left(|i| < |i'|\right) \ \mbox{  {\rm or}  } \ \left(|i|=|i'| \ \mbox{ {\rm and}  } \ |j| \leq |j'|\right),$$
with the convention that $\omega_i \prec_{\rm abslex} \omega_j$ means $\omega_i \preceq_{\rm abslex} \omega_j$ and $\omega_i \neq \omega_j$.

\begin{remark}
Note that the relation $\preceq_{\rm abslex}$ {\em is not} an order relation on the full alphabet $\Sigma^B_{(i)}$, due to the simple example of a pair of elements $(-1,1)$ and $(1,1)$, which satisfy $(-1,1) \preceq_{\rm abslex} (1,1)$ as well as $(1,1) \preceq_{\rm abslex} (-1,1)$. This is the reason we define this order over each RG-word of type $B$ of the first kind separately, where two such pairs can not appear simultaneously. 
\end{remark}

\medskip

Based of this order, we define three natural statistics on the set of RG-words of type $B$ of the first kind: 

\begin{definition}
\label{def finv}
Let $\omega=\omega_1\cdots \omega_n\in R_{(i)}^B(n,k)$. Define the {\em inversion of $\omega$}, denoted ${\rm inv}_B$, by $${\rm inv}_B(\omega)=\#\left\{(\omega_i,\omega_j) \mid i<j,\ \omega_j \prec_{\rm abslex} \omega_i \right\},$$
and the size of the set of negative elements:
${\rm neg}(\omega)=\#\left\{\omega_t=(i_t,j_t) \mid j_t<0\right\}$.

Furthermore, define the {\em flag-inversion} of an RG-word $\omega$ as follows: 
$${\rm finv}(\omega)=2 {\rm inv}_B(\omega)+{\rm neg}(\omega).$$

\end{definition}

The name of the parameter {\it flag-inversion} is based on its similarity to the flag major index defined by Adin and Roichman \cite{AR}.

\begin{example}\label{example of RG1 word}
The signed permutation written in standard form (where the non-split cycles are emphasized)
$$\pi=\underbrace{(1,-7)}_{C_1}\underbrace{(-1,7)}_{-C_1} \underbrace{\mathbf{(2,-5,4,-9,-2,5,-4,9)}}_{C_2=-C_2}\underbrace{\mathbf{(3,8,-3,-8)}}_{C_3=-C_3}\underbrace{\mathbf{(6,-6)}}_{C_4=-C_4},$$
corresponds to the RG-word of type $B$ of the first kind: 
$$\omega=\underbrace{(1,1)}_{\omega_1}\underbrace{(-2,1)}_{\omega_2}\underbrace{(-3,1)}_{\omega_3}\underbrace{(-2,3)}_{\omega_4}\underbrace{(-2,-2)}_{\omega_5}\underbrace{(-4,1)}_{\omega_6}\underbrace{(1,-2)}_{\omega_7}\underbrace{(-3,2)}_{\omega_8}\underbrace{(-2,-4)}_{\omega_9}.$$
Then we have:
$${\rm inv}_B (\omega)= \#\left\{
\begin{array}{c}
(\omega_2,\omega_7), (\omega_3,\omega_4),(\omega_3,\omega_5),(\omega_3,\omega_7),(\omega_3,\omega_9),
(\omega_4,\omega_5), \\
(\omega_4,\omega_7),
(\omega_5,\omega_7),(\omega_6,\omega_7),
(\omega_6,\omega_8),(\omega_6,\omega_9),
(\omega_8,\omega_9)    
\end{array}\right\}=12 $$
and 
${\rm neg}(\omega)=\#\left\{\omega_5,\omega_7,\omega_9\right\} = 3$.
Therefore: 
$${\rm finv}(\omega)=2{\rm inv}_B(\omega)+{\rm neg}(\omega)=27.$$

\end{example}

\begin{remark}
The statistics ${\rm inv}_B$, ${\rm neg}$ and ${\rm finv}$ can also be considered as statistics on signed permutations. 
Let $\pi=C_1 \cdots C_k\in B_n$ be written in the standard form of the decomposition of $\pi$ in cycles and let $\tau$ be the signed permutation of $B_n$, represented in its one-line notation, obtained by erasing $-C$ (in the case that $C$ is a split cycle) or its second half (in the case that $C$ is a non-split cycle) of $\pi$ and removing the parentheses. Then, ${\rm neg}(\pi)$ is the number of negative elements in $\tau$.

Moreover, let $|\tau|$ be the permutation of $S_n$ obtained by removing the negative signs in $\tau$. Then it is easy to observe that ${\rm inv}_B(\pi)={\rm inv}(|\tau|)$, where for $\sigma \in S_n$, $(i,j) \in {\rm inv}(\sigma)$ if $i<j$ but $\sigma_i > \sigma_j$.
\end{remark}

We illustrate this remark using the following example: 
\begin{example}\label{example of a signed permutation}
We repeat the computation of the statistics ${\rm inv}_B$ and ${\rm neg}$ of the signed permutation from Example \ref{example of RG1 word} using its cycle decomposition: given the signed permutation $$\pi=(1,-7)(-1,7)(2,-5,4,-9,-2,5,-4,9)(3,8,-3,-8)(6,-6),$$
we have that: 
$$\tau=1,-7,2,-5,4,-9,3,8,6.$$
Now, ${\rm neg}(\pi)=3$, since there are three negative elements in $\tau$.
Moreover: 
\begin{eqnarray*}
{\rm inv}_B (\pi)&= &{\rm inv}(|\tau|)= {\rm inv}(172549386)=\\
&=& \#\left\{
\begin{array}{c}
(2,3), (2,4),(2,5),(2,7),(2,8),(4,5), \\
(4,7),(5,7),(6,7),(6,8),(6,9),(8,9)    
\end{array}\right\}=12,
\end{eqnarray*}
and thus:
$${\rm finv}(\pi)=2{\rm inv}_B(\pi)+{\rm neg}(\pi)=27.$$

\end{example}

\medskip

Our next step is using the statistic ${\rm finv}$ in defining a $q$-analogue of the Stirling number of type $B$ of the first kind:

\begin{definition}\label{def q stirling type B first kind}
We define the {\em $q$-Stirling number $s_q^B(n,k)$ of type $B$ of the first kind} as follows:
$$s_q^B(n,k) :=\sum\limits_{\omega \in R_{(i)}^B(n,k)}{q^{{\rm finv}(\omega)}}=\sum\limits_{\omega \in R_{(i)}^B(n,k)} q^{2{\rm inv}_B (\omega)+{\rm neg}(\omega)}.$$
\end{definition}

By the bijection between the set $R_{(i)}^B(n,k)$ and the group of signed permutations $B_n$, we have for $q=1$:
$$s_q^B(n,k)|_{q=1}=s^B(n,k).$$

As mentioned in Proposition \ref{recursion theorem first kind}, the Stirling number of type $B$ of the first kind $s^B(n,k)$ satisfies the following recursion:
$$s^B(n,k)=s^B(n-1,k-1)+(2n-1)s^B(n-1,k).$$

We now prove a Stirling-type recursion for $s_q^B(n,k)$:

\begin{prop}\label{recursion q version}
For each $1 \leq k\leq n$,
\begin{equation}
s^B_q(n,k)=s^B_q(n-1,k-1)+\left(1+[2n-2]_q\right) \cdot s^B_q(n-1,k),
\label{recursion q stirling B}
\end{equation}
with the boundary conditions: $$s^B_q(n,0)=\sum\limits_{\ell=1}^{n}s_{q^2}^A(n,\ell)\cdot(1+q)^{n-\ell} \mbox{ \ for \ } n \geq 1,$$ and $s^B_q(0,k)= \delta_{0k}$.
\end{prop}

\begin{proof}
We start by proving the recursion. Let 
$$f:R_{(i)}^{B}(n,k) \rightarrow R_{(i)}^{B}(n-1,k) \cupdot R_{(i)}^{B}(n-1,k-1)$$ 
be the function defined by removing the last pair $\omega_n=(i_n,j_n)$, i.e.  $$f(\omega) =f(\omega_1\cdots\omega_n)=\omega_1\cdots\omega_{n-1},$$ 
where $\omega_m=(i_m,j_m)$ for $1 \leq m \leq n-1$.
For $\omega \in R_{(i)}^{B}(n,k)$, we have exactly two possibilities: 
\begin{itemize}
\item[(a)] If $|i_n|>\max\left\{|i_1|,\dots,|i_{n-1}|\right\}$ (so $j_n=1$) and $i_n<0$ (i.e. $n$ is located alone in a {\it non-split} cycle), then $f(\omega)\in R_{(i)}^{B}(n-1,k-1)$. Actually, the RG-words of type $B$ of the first kind satisfying this condition in $R_{(i)}^{B}(n,k)$ are in bijection with $R_{(i)}^{B}(n-1,k-1)$, so we have:
$${\rm finv} (f(\omega)) = 2{\rm inv}_B(f(\omega))+{\rm neg}(f(\omega))=2{\rm inv}_B(\omega)+{\rm neg}(\omega)={\rm finv}(\omega),$$
since $(i_t,j_t) \preceq_{\rm abslex} (i_n,1)$ for $1 \leq t <n$, and so there are no new inversions.

\item[(b)] Otherwise, we distinguish between two sub-cases: 

\begin{itemize}
\item[(1)] If $i_n>\max\left\{|i_1|,\dots,|i_{n-1}|\right\}$ (so $j_n=1)$ and $i_n>0$ (i.e. $n$ is located alone in a {\it split} cycle), then: $f(\omega)\in R_{(i)}^{B}(n-1,k)$, which yields: $$\hspace{-5pt}{\rm finv}(f(\omega)) = 2{\rm inv}_B(f(\omega))+{\rm neg}(f(\omega))=2{\rm inv}_B(\omega)+{\rm neg}(\omega)={\rm finv}(\omega),$$
since $(i_t,j_t) \preceq_{\rm abslex} (i_n,1)$ for $1 \leq t <n$. Note that the RG-words of type $B$ of the first kind satisfying this condition in $R_{(i)}^{B}(n,k)$ are in bijection with $R_{(i)}^{B}(n-1,k)$ and hence contribute the coefficient $1$ in the second summand.

\item[(2)] If $|i_n| \leq \max\left\{|i_1|,\dots,|i_{n-1}|\right\}$
(i.e. $n$ is inserted to one of the existing cycles), then $f(\omega)\in R_{(i)}^{B}(n-1,k)$. Since by Definition \ref{def_RGB_kind1}(1), we have:   $(i_1,j_1)=(\pm 1, 1) \preceq_{\rm abslex} (i_n,j_n)$ \ (i.e. $n$ can not be inserted before the first place, due to the standard form condition), the addition of the element $\omega_n =(i_n,j_n)$ at the $n$th place of $\omega$ can create up to $n-2$ new inversions, depending on the number of pairs $\omega_t=(i_t,j_t)$ (for $t<n$) satisfying $(i_n,j_n)\preceq_{\rm abslex}(i_t,j_t)$. Moreover, $j_n$ can be either positive or negative (corresponding to the insertion of $n$ or $-n$ to an existing cycle). If $j_n>0$, then the actual contribution of $(i_n,j_n)$ is twice the number of inversions created, while if $j_n<0$, then $(i_n,j_n)$ contributes twice the number of inversions plus $1$. 
Hence, we have: 
\begin{eqnarray}\label{eqn case b2}
\hspace{-20pt}{\rm finv}(\omega)=2{\rm inv}_B(\omega)+{\rm neg}(\omega)&=&2{\rm inv}_B(f(\omega))+2\ell+{\rm neg}(f(\omega)) +\varepsilon_n= \\
\nonumber&=& {\rm finv}(f(\omega)) +2\ell +\varepsilon_n,
\end{eqnarray}
where $\ell \in \{0,1,\dots,n-2\}$ is the number of inversions created by the pair $(i_n,j_n)$ and $$\varepsilon_n=\left\{ \begin{array}{lc}
 0    &  j_n>0 \\
1     & j_n<0
\end{array}\right. . $$

Note that in case $(i_n,j_n)$ does create an inversion, the pairs $(i_k,j_k)$ satisfying $i_k=i_n$ and $|j_k| \geq |j_n|$ (i.e. the pairs corresponding to the elements in the cycle containing $\pm n$ which are located after $\pm n$) will be modified to $(i_k,j_k+1)$, respectively, but this modification does not create any new inversion
(see Example \ref{example with table} below). 
Hence, the total contribution is the the coefficient $[2n-2]_q$ in the second summand.
\end{itemize}

\end{itemize}

Summing up all possibilities yields the requested recursion.

\medskip

Regarding the first boundary condition, the polynomial $s^B_q(n,0)$ is the generating function of RG-words of type $B$ of the first kind $\omega=\omega_1\cdots \omega_n=(i_1,j_1) \cdots (i_n,j_n),$ satisfying $i_t>0$  for each $1\leq t \leq n$ (i.e. the signed permutations having only split cycles), where each word $\omega$ contributes a weight $q^{{\rm finv}(\omega)}$. Since $i_t>0$ for each $1\leq t \leq n$, we can actually sum over the RG-words of type $A$ of the first kind, and possibly add signs to $j_t$ for pairs $(i_t,j_t)$ satisfying $|j_t|>1$ (i.e. we can sign each element of the cycle, except for the cycle's first element).

Recall that:
$s_q^A(n,\ell)=\sum\limits_{\omega \in R^A_{(i)}(n,\ell)}q^{\rm inv(\omega)}$ and let $\omega \in R^A_{(i)}(n,\ell)$ (see Definition \ref{definition of RGA}). Given\break $1\leq t\leq n$ such that $|j_t|>1$, for every $s$ such that $i_s=i_t$ (i.e. $s$ and $t$ are in the same cycle), $j_s$ can be either positive or negative. In total, we have $n-\ell$ pairs $\omega_t=(i_t,j_t)$ that can be signed. Hence, we have for $n \geq 1$:
$$s_q^B(n,0)=\sum\limits_{\omega \in R^B_{(i)}(n,0)} q^{{\rm finv}(\omega)}=\sum\limits_{\omega \in R^B_{(i)}(n,0)} q^{2{\rm inv}_B(\omega)+{\rm neg}(\omega)}=\sum\limits_{\ell=1}^{n}s_{q^2}^A(n,\ell) \cdot (1+q)^{n-\ell}.$$

The second boundary condition $s^B_q(0,k)= \delta_{0k}$ is a convention. 
\end{proof}

\begin{remark}
Note that recursion (\ref{recursion q stirling B}) is different from the recursion of the first kind $q$-Stirling number of type $B$ defined by Sagan and Swanson \cite{SaSw}. In their recursion, our coefficient 
$1+[2n-2]_q$ is replaced by $[2n-1]_q$. We will elaborate on their statistic in Section \ref{Sec SaSw parameter}. 

One should notice that the coefficient of the second summand in our recursion actually reflects the two different cases (cases (b)(1) and (b)(2) above) for adding the last pair $\omega_n$.  
\end{remark}

\begin{example} \label{example with table}
We show the contribution of the various possibilities of inserting $\omega_{6}=(i_6,j_6)$ into the RG-word of type $B$ of the first kind
$$\omega=\underbrace{(1,1)}_{\omega_1}\underbrace{(2,1)}_{\omega_2}\underbrace{(-3,1)}_{\omega_3}\underbrace{{(2,3)}}_{\omega_4}\underbrace{(2,-2)}_{\omega_5},$$
which corresponds to the signed permutation:
$$(1)\ (-1)\ (2,-5,4)\ (-2,5,-4)\ (3,-3).$$

\hspace{-40pt}
\begin{tabular}{|c|c|c|c|}
\hline
$\#$ & {\rm New RG-word} & {\rm New permutation presented by cycles} & {\rm Contribution} \\
\hline
$(1)$ & {\small $(1,1)\ (2,1)\ (-3,1)\ (2,3)\ (2,-2)\ {\bf (1,2)}$} & {\small $(1,{\bf 6})\ (-1,{\bf -6})\ (2,-5,4)\ (-2,5,-4)\ (3,-3)$} & {\small $q^8$}\\
\hline
$(2)$ & {\small $(1,1)\ (2,1)\ (-3,1)\ (2,3)\ (2,-2)\ {\bf (1,-2)}$} & {\small $(1,{\bf-6})\ (-1,{\bf 6})\ (2,-5,4)\ (-2,5,-4)\ (3,-3)$} & {\small $q^9$}\\
\hline
$(3)$ & {\small $(1,1)\ (2,1)\ (-3,1)\ (2,{\bf 4})\ (2,{\bf -3})\ {\bf (2,2)}$}&
{\small $(1)\ (-1)\ (2,{\bf6},-5,4)\ (-2,{\bf-6},5,-4)\ (3,-3)$} &{
\small $q^6$} \\
\hline
$(4)$ & {\small $(1,1)\ (2,1)\ (-3,1)\ (2,{\bf 4})\ (2,{\bf -3})\ {\bf (2,-2)}$} & {\small $(1)\ (-1)\ (2,{\bf-6},-5,4)\  (-2,{\bf 6},5,-4)\ (3,-3)$ } &{\small $q^7$} \\
\hline
$(5)$ & {\small $(1,1)\ (2,1)\ (-3,1)\ (2,{\bf 4})\ (2, -2)\ {\bf (2,3)}$} & {\small $(1)\ (-1)\ (2,-5,{\bf6},4)\ (-2,5,{\bf -6},-4)\ (3,-3)$ } &{\small $q^4$} \\
\hline
$(6)$ & {\small $(1,1)\ (2,1)\ (-3,1)\ (2,{\bf 4})\ (2, -2)\ {\bf (2,-3)}$} & {\small $(1)\ (-1)\ (2,-5,{\bf-6},4)\ (-2,5,{\bf 6},-4)\ (3,-3)$} & {\small $q^5$}\\
\hline
$(7)$ & {\small $(1,1)\ (2,1)\ (-3,1)\ (2,3)\ (2,-2)\ {\bf (2,4)}$} & {\small $(1)\ (-1)\ (2,-5,4,{\bf6})\ (-2,5,-4,{\bf-6})\ (3,-3)$} & {\small $q^2$}\\
\hline
$(8)$ & {\small $(1,1)\ (2,1)\ (-3,1)\ (2,3)\ (2, -2)\ {\bf (2,-4)}$}& {\small $(1)\ (-1)\ (2,-5,4,{\bf-6})\ (-2,5,-4,{\bf 6})\ (3,-3)$} & {\small $q^3$}\\
\hline
$(9)$ & {\small $(1,1)\ (2,1)\ (-3,1)\ (2,3)\ (2, -2)\ {\bf (-3,2)}$} & {\small $(1)\ (-1)\ (2,-5,4)\  (-2,5,-4)\ (3,{\bf6},-3,{\bf-6})$} & $1$\\
\hline
$(10)$ & {\small $(1,1)\ (2,1)\ (-3,1)\ (2,3)\ (2, -2)\ {\bf (-3,-2)}$} & {\small $(1)\ (-1)\ (2,-5,4)\ (-2,5,-4)\ (3,{\bf-6},-3,{\bf6})$} & $q$ \\
\hline
$(11)$ & {\small $(1,1)\ (2,1)\ (-3,1)\ (2,3)\ (2, -2)\ {\bf (-4,1)}$} & {\small $(1)\ (-1)\ (2,-5,4)\ (-2,5,-4)\ (3,-3)\ {\bf (6,-6)}$} & $1$\\
\hline
$(12)$ & {\small $(1,1)\ (2,1)\ (-3,1)\ (2,3)\ (2, -2)\ {\bf (4,1)}$} & {\small $(1)\ (-1)\ (2,-5,4)\ (-2,5,-4)\ (3,-3)\ {\bf (6)}\ {\bf (-6)}$} & $1$ \\
\hline
\end{tabular}

\medskip

Note that lines $(1)-(10)$ correspond to case (b)(2) in the proof of Proposition \ref{recursion q version} above, line $(11)$ corresponds to case (a), and line $(12)$ to case (b)(1).  
\end{example}


\subsubsection{The {\it q,r}-Stirling numbers of type {\it B}  of the first kind}

We define a $q,r$-analogue for the Stirling numbers  of type $B$ of the first kind, similarly to what we have done in type $A$ (see Section \ref{sec q,r-analogue type A} above).

\begin{definition}
The {\em $q,r$-Stirling number $s_{q}^B(n,k,r)$ of type $B$ of the first kind} is defined as follows:
$$s_{q}^B(n,k,r) := \sum\limits_{\omega \in R_{(i)}^{B}(n,k,r)} q^{{\rm finv} (\omega)}=\sum\limits_{\omega \in R_{(i)}^{B}(n,k,r)} q^{2{\rm inv}_B (\omega)+{\rm neg}(\omega)}.$$
\end{definition}

Next, we prove the following result regarding the $q,r$-analogue:

\begin{prop}\label{recursion q,r version}
For each $1 \leq r < k\leq n$,
\begin{equation}
s^B_{q}(n,k,r)=s^B_{q}(n-1,k-1,r)+\left(1+[2n-2]_q\right) \cdot s^B_{q}(n-1,k,r),
\label{recursion q,r stirling B}
\end{equation}
with the boundary conditions: $s_{q}^B(r,r,r)= 1,$
$s_{q}^B(n,k,r)=0$ \ \ for \ \ $0\leq k < r$, and $s^B_{q}(0,k,r)= \delta_{0k}$.
\end{prop}

\begin{proof}
We proceed directly to the boundary conditions, since the recursion is identical to the one given in Proposition \ref{recursion q version} above. 

Regarding the first boundary condition, the polynomial $s^B_{q}(r,r,r)$ is the generating function of the RG-word of type $B$ of the first kind $\omega=\omega_1\cdots \omega_r=(-1,1) \cdots (-r,1),$  
which is a single word contributing a weight $q^{0}=1$. 

The other boundary conditions $s_{q}^B(n,k,r)=0$ for $0\leq k < r$ and $s^B_q(0,k,r)= \delta_{0k}$ are conventions. 
\end{proof}

\subsubsection{{\it q}- and {\it q,r}-analogues of the presentation of Stirling numbers as transition coefficients}

Recall that Equation (\ref{generalized first kind def}) presents the Stirling numbers of type $B$ of the first kind as 
transition coefficients between two bases. Here we prove a $q$-analogue of Equation (\ref{generalized first kind def}) in a combinatorial way:

\begin{prop}\label{generating function of q analogue}
\begin{equation}
(t+1)\left(t+1+[2]_q\right)\left(t+1+[4]_q\right)\cdots \left(t+1+[2n-2]_q\right)=\sum\limits_{k=0}^n{s_q^B(n,k)\ t^k}\ .
\label{generalized first kind def q}
\end{equation}
\end{prop}

\begin{proof}
By Definition \ref{def q stirling type B first kind}, the coefficient of $t^k$ on the right-hand side of Equation (\ref{generalized first kind def q}) is given by collecting the contributions of the parameter ${\rm finv}(\omega)$ running through all the RG-words\break $\omega=(i_1,j_1) \cdots (i_n,j_n)$ of type $B$ of the first kind, such that there are exactly $k$ indices $m_1,\dots,m_k$ whose corresponding pairs satisfy $i_{m_u}<0$ and $j_{m_u}=1$ for $1 \leq u \leq k$, i.e. there are $k$ non-split cycles in the corresponding signed permutation.  

On the other hand, after expanding the left-hand side, we get an expression of the form $\sum\limits_{k=0}^n c_k(q) t^k$, where for each $k$, the coefficient $c_k(q)$ is a sum of products of the form $$1 \cdot \left(1+[2]_q\right) \cdot \left(1+[4]_q\right) \cdots \left(1+[2n-2]_q\right),$$ such that exactly $k$ terms from this product are missing. 

For each single product of this form, denote the $k$ missing terms by $$1+[a_1]_q\ ,\ \dots \ , \ 1+[a_k]_q$$ (where possibly $a_1=0$ and so $1+[0]_q=1$). 

Denote $\ell_u = \frac{a_u}{2}+1$ for $1 \leq u \leq k$ (note that $a_u$ is always even). Then, as will be demonstrated in Example \ref{example of proof for first kind} below, such a product is the sum of contributions of the parameter ${\rm finv}(\omega)$ of RG-words  
$\omega_1 \omega_2 \cdots \omega_n = (i_1,j_1)\ (i_2,j_2) \cdots (i_n,j_n)$ of type $B$ of the first kind 
having $i_{\ell_u}<0$ and $j_{\ell_u} = 1$ for all $1 \leq u \leq k$ (i.e. $\ell_u$ is the smallest element in a non-split cycle).
\end{proof}

\begin{example}\label{example of proof for first kind} For $n=8$, one of the summands contributing to the coefficient of $t^3$ in the product $(t+1)(t+1+[2]_q)\cdots (t+1+[14]_q)$ is: $$1 \cdot (1+[4]_q) \cdot (1+[6]_q) \cdot (1+[10]_q) \cdot (1+[14]_q)$$ (with three terms missing, corresponding to $a_1=2,\ a_2=8,\ a_3=12$). Therefore: $$\ell_1=\frac{a_1}{2}+1=2, \quad \ell_2=\frac{a_2}{2}+1=5 \ \mbox{  and }\ \ell_3=\frac{a_3}{2}+1=7.$$ 
We construct all RG-words $\omega_1 \cdots \omega_8=(i_1,j_1) \cdots (i_8,j_8)$
of type $B$ of the first kind satisfying 
$$i_2<0,\ i_5<0,\ i_7<0, \qquad j_2=j_5=j_7=1$$ 
and there is no other pair $(i_m,1)$ with $i_m<0$ for $m\neq 2,5,7$. 

This set of RG-words of type $B$ of the first kind corresponds to the set of signed permutations having three non-split cycles, such that
$2, 5, 7$ are the minimal elements of these non-split cycles.

\medskip

Now, we compute the contribution of the parameter ${\rm finv}(\omega)$ induced by the addition of the other pairs into the various possible RG-words:

\begin{itemize} 
\item The first pair $(i_1,j_1)$ must be $(1,1)$ by the requirements above, and therefore has no contribution. This corresponds to the number $1$ which is located as the minimal element of a split cycle. 

\item The second pair must be  $(i_2,j_2)=(-2,1)$, corresponding to $\ell_1=2$, and also has no contribution. 
This corresponds to the number $2$ which is located as the minimal element of a non-split cycle. 

\item The third pair $(i_3,j_3)$ can be one out of the following $5$ possibilities:
$$(1,2), (1,-2),(-2,2),(-2,-2),(3,1)$$
(corresponding to inserting the number $3$ in the second place of the first split cycle  as positive or negative element, in the second place of the second non-split cycle as positive or negative element, or in a new split cycle).
The different contributions of the various possibilities of the third pair are:  
$q^2,q^3,1,q,1$, respectively, whose sum is the term $1+[4]_q$.

\item The fourth pair $(i_4,j_4)$ can be one out of\ \ $7$ possibilities, corresponding to inserting the number $4$ after $1$ as positive or negative element, after $2$ as a positive or negative element, after $3$ as a positive or negative element, or in a new split cycle (here we cannot write it directly in the language of RG-words, since it depends on the choice of the pair $(i_3,j_3)$).
Nevertheless, the different contributions of the various possibilities of the forth pair sum to  
the term $1+[6]_q$.

\item Continuing in this manner, one can see that the total contribution to the generating function of the RG-words satisfying the requirements above is 
$1 \cdot (1+[4]_q) \cdot (1+[6]_q) \cdot (1+[10]_q) \cdot (1+[14]_q)$ as needed.
\end{itemize}

\end{example}

\begin{remark}
Note that a similar argument was used by Benjamin and Quinn \cite[p. 96]{BQ} and by the authors of the current paper \cite[Theorem 4.1]{monthly} in a different context.     
\end{remark}

Combining the first boundary condition of Proposition \ref{recursion q version} with the free coefficient  of Equation (\ref{generalized first kind def q})  yields the following identity: 
\begin{cor}
$$1 \cdot \left(1+[2]_q\right) \cdot \left(1+[4]_q\right) \cdots \left(1+[2n-2]_q\right)=\sum\limits_{\ell=1}^{n}s_{q^2}^A(n,\ell) \cdot (1+q)^{n-\ell}.$$
\end{cor}

\medskip

Now, we pass to the $r$-variant:

\begin{prop}\label{generating function of q-r analogue prop}
\begin{equation}\label{generalized first kind def q,r}
\left(t+1+[2r]_q\right)\left(t+1+[2r+2]_q\right)\cdots \left(t+1+[2n-2]_q\right)=\sum\limits_{k=0}^{n-r}{s_{q}^B(n,r+k,r)\ t^k}\ .
\end{equation}
\end{prop}

\begin{proof}
Multiplying Equation (\ref{generalized first kind def q,r}) by $t^r$ yields:
\begin{equation}
t^r\left(t+1+[2r]_q\right)\left(t+1+[2r+2]_q\right)\cdots \left(t+1+[2n-2]_q\right)=\sum\limits_{k=r}^{n}{s_{q}^B(n,k,r)\ t^k}\ .
\label{generalized first kind def q,r 1}
\end{equation}

The proof of this equality is very similar to the proof of Proposition \ref{generating function of q analogue}, see Example \ref{example of proof for r first kind} below. \end{proof}

\begin{example}\label{example of proof for r first kind} We continue with $n=8$, $k=4$ and $r=3$. One of the summands contributing to the coefficient of $t^4$ in the product $t^3(t+1+[6]_q)\cdots (t+1+[14]_q)$ is: $$ (1+[6]_q) \cdot (1+[8]_q) \cdot (1+[10]_q)\cdot (1+[14]_q)$$ (with one term missing, corresponding to $a_1=12$). Therefore: $\ell_1=\frac{a_1}{2}+1=7.$ 
We construct all the RG-words of type $B$ of the first kind $\omega_1 \cdots \omega_8=(-1,1)(-2,1)(-3,1)(i_4,j_4) \cdots (i_8,j_8)$ satisfying 
$i_7<0$ and $j_7=1$, 
and there is no other pair $(i_m,1)$ with $i_m<0$ for $m\neq 1,2,3,7$. 

This set of RG-words of type $B$ corresponds to the set of signed permutations having four non-split cycles, such that
$1,2, 3, 7$ are their corresponding minimal elements.

\medskip

Now, we compute the contribution of the parameter ${\rm finv}(\omega)$ induced by the addition of the other pairs into the various possible RG-words, where we start from the fourth pair, since the first three pairs are already determined:

\begin{itemize} 
\item The fourth pair $(i_4,j_4)$ can be one out of\ \ $7$ possibilities, corresponding to inserting the number $4$ after $1$, $2$ or $3$ as a positive or a negative element, or in a new split cycle.
The different contributions of the various possibilities sum to  
the term $1+[6]_q$.

\item The fifth pair $(i_5,j_5)$ can be one out of\ \ $9$ possibilities, corresponding to inserting the number $5$ after $1$, $2$, $3$ or $4$ as a positive or a negative element, or in a new split cycle.
The different contributions of the various possibilities  sum to  
the term $1+[8]_q$.

\item Continuing in this manner, one can see that the total contribution to the generating function of the RG-words satisfying the requirements above is 
$1 \cdot (1+[6]_q) \cdot (1+[8]_q) \cdot (1+[10]_q) \cdot (1+[14]_q)$ as needed.
\end{itemize}

\end{example}

\section{Additional variants of the statistic finv}
\label{sec additional variant finv}

A drawback of the parameter finv 
over signed permutations that was defined in Definition \ref{def finv} is that the coefficient $1+[2n-2]_q$ in the recursion (\ref{recursion q stirling B}) and in the generating function in Equation (\ref{generalized first kind def q}) is less natural than the ordinary coefficient $[2n-1]_q$, which appears more often in the literature. Moreover, this ordinary coefficient will be needed in order to get orthogonality and sum of powers results in Section \ref{sec orthogonality} below.  

We suggest two different solutions to this problem: the first is a slight variation of our parameter which fixes the coefficient (see Section \ref{sec slight variation}) and the other is a reformulation of a different parameter which has been presented first by Sagan and Swanson \cite{SaSw} (see Section \ref{Sec SaSw parameter}). 

\subsection{A slight variation of the finv parameter}\label{sec slight variation}

In order to switch from the coefficient  
$1+[2n-2]_q$ to the coefficient $[2n-1]_q$ 
in the recursion (\ref{recursion q stirling B}) and in the generating function in Equation (\ref{generalized first kind def q}), we define here the following auxiliary parameter:
\begin{definition}\label{def sfinv}
Let $\omega=\omega_1\cdots \omega_n\in R_{(i)}^B(n,k)$. Define 
$${\rm nl}(\omega)=\#\left\{\omega_t = (i_t,j_t)\mid j_t \neq 1 \right\}.$$
 
Define the {\em shifted flag-inversion} of an RG-word $\omega$ as follows: $${\rm sfinv}(\omega):={\rm finv}(\omega)+{\rm nl}(\omega)
=2 {\rm inv}_B(\omega)+{\rm neg}(\omega)+{\rm nl}(\omega).$$
\end{definition}

In the language of signed permutations decomposed into cycles, if $\pi \in B_n$ is the signed permutation corresponding to the RG-word  $\omega$, then the parameter ${\rm nl}(\omega)$ equals  $n$ minus the number of representative cycles (of both types) in $\pi$. Adding this parameter to finv actually contributes an additional multiplication by $q$ to each element which is not minimal in its cycle (corresponding to case (b)(2) inside the proof of Proposition \ref{recursion q version}), and hence `corrects' the coefficient 
$1+[2n-2]_q$ to the coefficient $[2n-1]_q$.

\begin{example}\label{example sfinv}
(a) Given the signed permutation of Example \ref{example of a signed permutation}
$$\pi=(1,-7)(-1,7)(2,-5,4,-9,-2,5,-4,9)(3,8,-3,-8)(6,-6),$$
where we got \ ${\rm finv}(\omega)=27$.
Note that ${\rm nl}(\omega)=9-4=5$, so we have ${\rm sfinv}(\omega)=27+5=32$.

\medskip

\noindent
(b) In Example \ref{example with table}, we listed the changes in the contribution of the various possibilities of inserting $\omega_{6}=(i_6,j_6)$ into the RG-word of type $B$ of the first kind
$$\omega=\underbrace{(1,1)}_{\omega_1}\underbrace{(2,1)}_{\omega_2}\underbrace{(-3,1)}_{\omega_3}\underbrace{{(2,3)}}_{\omega_4}\underbrace{(2,-2)}_{\omega_5},$$
which corresponds to the signed permutation:
$$(1)\ (-1)\ (2,-5,4)\ (-2,5,-4)\ (3,-3).$$

Using the new parameter, the contribution of the first 10 lines of the table appearing in this example will be multiplied by $q$, since the new additional element $6$ is placed in a non-leading location. Therefore, the sum of the contributions will be the coefficient 
$1+[11]_q$, as needed.
\end{example}

This leads us to the following definition of the $q$-shifted-Stirling number:
\begin{definition}\label{def q shifted stirling type B first kind}
We define the {\em $q$-shifted-Stirling number $ss_q^B(n,k)$ of type $B$ of the first kind} as follows:
$$ss_q^B(n,k) :=\sum\limits_{\omega \in R_{(i)}^B(n,k)}{q^{{\rm sfinv}(\omega)}}=\sum\limits_{\omega \in R_{(i)}^B(n,k)} q^{2{\rm inv}_B (\omega)+{\rm neg}(\omega)+{\rm nl}(\omega)}.$$
\end{definition}

Then we have the following recursion:
\begin{prop}
For each $1 \leq k\leq n$,
\begin{equation}
ss^B_q(n,k)=ss^B_q(n-1,k-1)+[2n-1]_q \cdot ss^B_q(n-1,k),
\label{recursion q shifted stirling B}
\end{equation}
with the boundary conditions: $$ss^B_q(n,0)=[1]_q [3]_q \cdots [2n-1]_q$$ and $ss^B_q(0,k)= \delta_{0k}.$
\end{prop}

\begin{proof}
The proof of the recursion goes verbatim as the proof of the recursion in Proposition \ref{recursion q version}, except for the case (b)(2) which leads to Equation (\ref{eqn case b2}), that in our context should be replaced by the following equation:
\begin{eqnarray*}
{\rm sfinv}(\omega)&=&2{\rm inv}_B(\omega)+{\rm neg}(\omega) + {\rm nl} (\omega)=\\
&=&2{\rm inv}_B(f(\omega))+2\ell+{\rm neg}(f(\omega)) +\varepsilon_n  + {\rm nl}(f(\omega))+1= \\
&=& {\rm sfinv}(f(\omega)) +2\ell +\varepsilon_n+1,
\end{eqnarray*}
and hence the corresponding coefficient  in the recursion changes to $[2n-1]_q$.

\medskip

The first boundary condition is based on  Proposition \ref{generating function of q analogue for the variation} stated below, for the case $k=0$, whose proof is independent of the current result.
The second boundary condition $s^B_q(0,k)= \delta_{0k}$ is a convention. 
\end{proof}

The presentation of the variation of our $q$-analogue $ss_q^B(n,k)$ as transition coefficients between two bases takes the following form, the proof of which is very similar to the proof of Proposition \ref{generating function of q analogue} above:

\begin{prop}\label{generating function of q analogue for the variation}

\begin{equation}\label{eqn generating function of q analogue for the variation}
(t+1)\left(t+[3]_q\right)\left(t+[5]_q\right)\cdots \left(t+[2n-1]_q\right)=\sum\limits_{k=0}^n{ss_q^B(n,k)\ t^k}\ .    
\end{equation}
\end{prop}

\subsection{The Sagan-Swanson inversion parameter}\label{Sec SaSw parameter}

Sagan and Swanson \cite[Definition 3.15]{SaSw} presented a different $q$-analogue for the Stirling number of type $B$ of the first kind, whose generating function is the left-hand side of Equation (\ref{eqn generating function of q analogue for the variation}) and satisfies recursion (\ref{recursion q shifted stirling B}).   
Their parameter was defined in the context of signed permutations as the number of inversions of the full presentation as a product of signed cycles using their standard form. 
After presenting their standard form and parameter in Definition \ref{def SaSw parameter}, we suggest a different view of this parameter in the style of flag-statistics (see e.g. \cite{AR}), by using the representative cycles defined above and decomposing the set of inversions into a sum of parameters, based on this presentation. 

\begin{definition}[Sagan-Swanson's parameter \cite{SaSw}]\label{def SaSw parameter}
Given a signed permutation $\pi=c_1 \cdots c_k$, the Sagan-Swanson standard form of $\pi$, written in its full presentation, satisfies the following conditions, where $m_i= \min |c_i|$ \ for each $1\leq i \leq k$:
\begin{enumerate}
\item $m_1 \leq m_2 \leq \cdots \leq m_k$,
\item If $m_i=m_{i+1}$, then $-m_i\in c_i$ and $m_i \in c_{i+1}$,
\item Each $c_i$ is written with its $\pm m_i$
in the last position, with non-split cycles ending in $-m_i$. 
\end{enumerate}

Note that Sagan and Swanson use the terms `paired' and 'unpaired' for what we call here `split' and `non-split' respectively. 

\medskip

In this context, they let $w = w_1w_2 \cdots w_{2n}$ be the
word obtained by removing the parentheses from the standard form of $\pi$, and define the set of inversions of $\pi$ to be
$${\rm Inv} (\pi) = \{(i, j) \mid i < j \mbox { and } w_i > |w_j|\},$$
with the corresponding inversion number
${\rm inv}(\pi)  = \# \ {\rm Inv} (\pi)$.
\end{definition}

\begin{example}\label{example SaSw inv part 1}
The signed permutation we encountered in Example \ref{example of a signed permutation}
$$\pi=(1,-7)(-1,7)(2,-5,4,-9,-2,5,-4,9)(3,8,-3,-8)(6,-6),$$
is written in Sagan-Swanson's standard form as
$$\pi=(7,-1)(-7,1)(5,-4,9,2,-5,4,-9,-2)(-8,3,8,-3)(6,-6).$$
Then we have: $$w =\underbrace {7}_{(1)},\underbrace{-1}_{(2)},\underbrace{-7}_{(3)},\underbrace{1}_{(4)},\underbrace{5}_{(5)},\underbrace{-4}_{(6)},\underbrace{9}_{(7)},\underbrace{2}_{(8)},\underbrace{-5}_{(9)},\underbrace{4}_{(10)},\underbrace{-9}_{(11)},\underbrace{-2}_{(12)},\underbrace{-8}_{(13)},\underbrace{3}_{(14)},\underbrace{8}_{(15)},\underbrace{-3}_{(16)},\underbrace{6}_{(17)},\underbrace{-6}_{(18)}, $$
and thus $${\rm Inv}(\pi)=\left\{\begin{array}{l}\underbrace{\textcolor{red}{(1,2),(1,4)}}_{(b)},\underbrace{\textcolor{blue}{(1,5),(1,9),(1,6),(1,10),(1,8),(1,12),(1,14),(1,16)}}_{(a)},\\
\underbrace{\textcolor{blue}{(1,17),(1,18)}}_{(a)},\underbrace{\textcolor{red}{(5,6),(5,10),(5,8),(5,12)}}_{(b)},\underbrace{\textcolor{blue}{(5,14),(5,16)}}_{(a)},\\
\underbrace{\textcolor{red}{(7,8),(7,12)}}_{(b)},\underbrace{(7,9),(7,10)}_{(d)},
\underbrace{\textcolor{blue}{(7,13),(7,15),(7,14),(7,16),(7,17),(7,18)}}_{(a)},\\
\underbrace{\textcolor{blue}{(10,14),(10,16)}}_{(a)},
\underbrace{\textcolor{purple}{(10,12),(15,16)}}_{(c)},\underbrace{\textcolor{blue}{(15,17),(15,18)}}_{(a)}\end{array} \right\}$$
so that ${\rm inv}(\pi)=34$
(where the meaning of the division into different categories will be explained later).
\end{example}

We want to present Sagan-Swanson's parameter ${\rm inv}$ as a flag-statistic, and for that 
we define the {\it shortened form} of (the standard form of) $\pi$, to be the list obtained by taking only the first half in each cycle (which can be considered as the representative cycle of this presentation). Denoting the resulting word by $\sigma$, we consider also its absolute word $|\sigma|$, which is the word obtained from $\sigma$ by forgetting the signs of its elements. 

In our running example (see Example \ref{example SaSw inv part 1}), we have: 
$$\sigma=7,-1,5,-4,9,2,-8,3,6 \ \ \ \mbox{   and   } \ \ \ |\sigma|=7,1,5,3,9,2,8,3,6.$$ 

\medskip

Now let us define the following four subsets:

\begin{itemize}
\item $A= \{ (i,j) : i<j, \ |\sigma|(i) > |\sigma|(j), \mbox{ $\sigma(i)$ and $\sigma(j)$ are in different cycles of $\pi$}\}$,

\item $B= \{ (i,j) : i<j, \ |\sigma|(i) > |\sigma|(j), \mbox{ $\sigma(i)$,$\sigma(j)$ are in the same cycle of $\pi$ }, \sigma(i)>0 \}$,

\item $C= \{ (i,j) : i<j, \ |\sigma|(i) > |\sigma|(j), \mbox{ $\sigma(i)$,$\sigma(j)$ are in the same cycle of $\pi$ }, \sigma(i)<0 \}$,

\item $D= \{ (i,j) : i<j, \ \sigma(i) < \sigma(j), \mbox{ $\sigma(i)$,$\sigma(j)$ are in the same cycle of $\pi$} \}$.
\end{itemize}

Based on these subsets, define the following four parameters:
$$p_A= |A|,\ p_B= |B|,\ p_C= |C|,\ p_D= |D|.$$

\medskip

We have the following result, that presents Sagan-Swanson's parameter inv as a flag-statistic:
\begin{prop}
Let $\pi$ be a signed permutation. Then:
 $${\rm inv}(\pi)=2(p_A+p_B)+(p_C+p_D).$$   
\end{prop}

\begin{proof}
Define the multiset $M(\pi)$, whose elements are two copies of the elements of the set $A$, two copies of the elements of the set $B$, and the elements of the sets $C$ and $D$. 
The proof will be complete as soon as we point out a bijection between the sets $M(\pi)$ and ${\rm Inv}(\pi)$.  

\medskip

We start by showing that each element of $M(\pi)$ corresponds to an element of ${\rm Inv}(\pi)$. We divide in four cases according to the sets $A,B,C,D$:
\begin{itemize}
\item If $(i,j)\in A$, then $|\sigma|(i)>|\sigma|(j)$ where $\sigma(i)$ and $\sigma(j)$ are in difference cycles of $\pi$. In this case, exactly one of the instances of $\sigma(i)$ is positive, and it appears before the positive and negative instances of $|\sigma|(j)$, and thus contributes two inversions to ${\rm Inv}(\pi)$ (e.g. pairs $(1,5)$ and $(1,9)$ in Example \ref{example SaSw inv part 1} above). 

\item If $(i,j)\in B$, then $|\sigma|(i)>|\sigma|(j)$ where $\sigma(i)$ and $\sigma(j)$ are in the same cycle of $\pi$, and $\sigma(i)>0$. In this case,  the positive instance of $\sigma(i)$  appears again before the positive and negative instances of $|\sigma|(j)$ (since one instance of $|\sigma|(j)$ appears after $\sigma(i)$ and the second instance appears in the deleted half of the cycle). Thus, it contributes two inversions to ${\rm Inv}(\pi)$ (e.g. pairs $(1,2)$ and $(1,4)$ in Example \ref{example SaSw inv part 1} above). 

\item If $(i,j)\in C$, then $|\sigma|(i)>|\sigma|(j)$ where $\sigma(i)$ and $\sigma(j)$ are in the same cycle of $\pi$, and $\sigma(i)<0$. In this case,  the negative instance of $\sigma(i)$ does not contribute to ${\rm Inv}(\pi)$, but its positive instance (appearing in the deleted half of the cycle) appears before an instance of  $|\sigma|(j)$, and thus contributes one inversion to ${\rm Inv}(\pi)$ (e.g. the pair $(10,12)$ in Example \ref{example SaSw inv part 1} above). 

\item If $(i,j)\in D$, then $\sigma(i)<\sigma(j)$ where $\sigma(i)$ and $\sigma(j)$ are in the same cycle of $\pi$. In this case, we have to split our treatment into four sub-cases (note that the case $\sigma(i)>0$ and $\sigma(j)<0$ is impossible, since by the assumption of this case we have $\sigma(i)<\sigma(j)$):
\begin{itemize}
\item If $\sigma(i)>0$ and $\sigma(j)>0$, then $\sigma(j)>|\sigma(i)|$, and thus $\sigma(j)$ appearing in the first half and $-\sigma(i)$ appearing afterwards in the deleted half contribute one inversion to ${\rm Inv}(\pi)$ (e.g. the pair $(7,9)$ in Example \ref{example SaSw inv part 1} above).  

\item If $\sigma(i)<0$ and $\sigma(j)<0$, then we have $|\sigma(j)|<|\sigma(i)|$. Therefore, in the deleted half of the cycle, $-\sigma(i)>0$ appears before   $-\sigma(j)>0$, and thus contribute one inversion to ${\rm Inv}(\pi)$.

\item If $\sigma(i)<0$, $\sigma(j)>0$ and $|\sigma(i)|<|\sigma(j)|=\sigma(j)$, then $\sigma(j)$ appearing in the first half and $-\sigma(i)$ appearing afterwards in the deleted half contribute one inversion to ${\rm Inv}(\pi)$ (e.g. the pair $(7,10)$ in Example \ref{example SaSw inv part 1} above).   

\item If $\sigma(i)<0$, $\sigma(j)>0$ and $|\sigma(i)|>|\sigma(j)|$, then in the deleted half of the cycle, we have $-\sigma(i) > -\sigma(j)$, so this situation contributes one inversion to ${\rm Inv}(\pi)$.  
\end{itemize}
\end{itemize}

Now we show that any other pairs $(i,j)$, which do not satisfy any of the conditions defining the subsets $A,B,C,D$, do not contribute any inversions to ${\rm Inv} (\pi)$.
We divide our treatment into two cases:
\begin{itemize}
\item If the pair $(i,j)
$ satisfies that $|\sigma|(i) < |\sigma|(j)$, where $\sigma(i)$ and $\sigma(j)$ are in difference cycles of $\pi$ (e.g. the pairs $(2,5)$ and $(2,9)$ in Example \ref{example SaSw inv part 1} above), then it does not contribute to ${\rm Inv} (\pi)$, since the definition of ${\rm Inv} (\pi)$ takes into account the absolute value of the second element in the pair.
\item If the pair $(i,j)$ satisfies that $\sigma(i)$ and $\sigma(j)$ are in the same cycle, $|\sigma|(i)<|\sigma|(j)$ but $\sigma(i)>\sigma(j)$, then we have $\sigma(i)<0$ and $\sigma(j)<0$. By the definition of ${\rm Inv}(\pi)$, there is no contribution in this case. 
\end{itemize}

\end{proof}

\begin{remark}
Note that Sagan-Swanson's parameter is indeed different from any of our parameters, as for example the signed permutation $\pi$ appearing in Example  \ref{example of a signed permutation}, we have ${\rm inv}(\pi)=34$ (see Example \ref{example SaSw inv part 1}), whereas we have: ${\rm finv}(\pi)=27$ and ${\rm sfinv}(\pi)=32$ (see Example \ref{example sfinv}(a)).    
\end{remark}

\section{Orthogonality and sum of powers}\label{sec orthogonality}

In this section, we look at two identities involving both kinds of Stirling numbers of type $B$. The first identity is some sort of orthogonality property between the two arrays of numbers (see Graham-Knuth-Patashnik \cite[Table 250]{GKP}), while the second is a Faulhaber sum a.k.a. a power sum, see \cite{GZ05,Knuth93,merca15,merca18}. 
Their analogue to type $A$ appears as Theorems 3.16 and 3.17 in Egge's book \cite{Egge}: 
\begin{thm} \ \\

\begin{itemize}
\item[]  Thm. 3.16:\ \  
$\sum\limits_{j=0}^m (-1)^{j}
s( n+1,n+1 - j ) S (n + m - j,n) =0.$

\item[] Thm. 3.17:\ \ 
$1^m + 2^m + \cdots +n^m=
\sum\limits_{j=1}^m (-1)^{j-1}
j \cdot s( n+1 , n+1 - j) S( n + m - j, n).$
\end{itemize}
\label{Egges thm}
\end{thm}

The corresponding identities for type $B$ are the main results of this section, which will be proved here using symmetric functions:

\begin{thm}\label{orthogonality}
For all $m, n \geq 1$, we have: 
\begin{enumerate}
\item $
\sum\limits_{j=0}^m (-1)^{j}
ss^B_q(n,n-j)
S^B_q(n+ m - j,n)=0.$

\item $\sum\limits_{j=1}^m (-1)^{j-1}
j \cdot ss^B_q(n,n - j) S^B_q( n+ m - j,n) =[1]_q^m + [3]_q^m + \cdots +[2n-1]_q^m.$
\end{enumerate}

\end{thm}

Note that Sagan and Swanson \cite[Cor. 2.7]{SaSw} proved a result, which is similar to part (1) of Theorem \ref{orthogonality} for the case $q=1$.  

\medskip

Recall that for each $k$, the {\it elementary symmetric polynomial}  $e_k(x_1,\dots, x_n)$, the {\it complete homogeneous symmetric polynomial} $h_k(x_1,\dots, x_n)$ and the {\it power sum symmetric polynomial} $p_k(x_1,\dots, x_n)$ are defined as follows (see e.g. Egge \cite{Egge} or Stanley \cite[Section 7]{EC2}): 
$$e_k(x_1, \dots, x_n)=\sum\limits_{1 \leq i_1 < \cdots < i_k\leq n} x_{i_1} \cdots x_{i_k}; \qquad h_k(x_1, \dots, x_n)=\sum\limits_{1 \leq i_1 \leq \cdots \leq i_k\leq n} x_{i_1} \cdots x_{i_k};$$
$$ p_k(x_1,\dots,x_n)=x_1^k +\cdots +x_n^k.$$

It is easily seen that the following recurrence relations hold for the elementary and homogeneous symmetric polynomials (see \cite[Proposition 3.1]{Egge}):
\begin{equation}e_k(x_1, \dots, x_n)=e_k(x_1, \dots, x_{n-1}) +x_n e_{k-1}(x_1, \dots, x_{n-1}),\label{rec_ek}\end{equation}
\begin{equation}h_k(x_1, \dots, x_n)=h_k(x_1, \dots, x_{n-1}) +x_n h_{k-1}(x_1, \dots, x_n).\label{rec_hk}\end{equation}

The ordinary generating functions for the elementary and complete homogeneous symmetric polynomials are respectively:
$$E_n(t)=\sum\limits_{k=0}^{n}e_k(x_1,\dots,x_n) t^k=\prod_{j=1}^{n}(1+x_jt)$$ and 
$$H_n(t)=\sum\limits_{k=0}^{n}h_k (x_1,\dots,x_n)t^k=\prod_{j=1}^{n}\frac{1}{(1-x_jt)},$$ from which we immediately 
conclude that $E_n(-t)$ and $H_n(t)$ are multiplicative inverses, a fact that can also be written for each $m \geq 1$ as 
\begin{equation}\label{convolution}
\sum\limits_{j=0}^m (-1)^j e_j (x_1,\dots,x_n) h_{m-j} (x_1,\dots,x_n)=0, 
\end{equation}
see e.g. Stanley \cite[Equation (7.13)]{EC2}. 

An expression of the power sum polynomials as a weighted convolution of elementary symmetric polynomials and complete homogeneous polynomials is also known. For each $m \geq 1$, we have:
\begin{equation}\label{weighted convolution}
p_m(x_1,\dots,x_n)=\sum\limits_{j=1}^m (-1)^{j-1}j \cdot e_j(x_1,\dots,x_n)h_{m-j}(x_1,\dots,x_n),
\end{equation}
a combinatorial proof of which can be found in Egge \cite[Theorem 3.2]{Egge}. 

\medskip
The following two lemmas, which are obtained by substituting the polynomials $[1]_q,[3]_q,\dots,[2n-1]_q$ in $e_k(x_1,\dots, x_n)$ and in $h_k(x_1,\dots,x_n)$, are the crux points of the proof of Theorem \ref{orthogonality}. 

\begin{lemma}\label{specialization in e}
$$e_k\left([1]_q,[3]_q,\dots, [2n-1]_q\right)=  ss^B_q(n,n-k)  .$$
\end{lemma} 

\begin{proof}
We prove it by induction on $n$, letting $0 \leq k \leq n$.

For $n=1$ and $k=0$, the left hand side is $e_0(1)=1$ and the right hand side is $ss^B_q(1,1) =1$, thanks to the unique signed permutation having a non-split cycle $(1,-1) \in B_1$, presented by the RG-word of type $B$ of the first kind $\omega=\omega_1= (-1,1)$ (see Definition \ref{def_RGB_kind1}) having\break ${\rm sfinv}(\omega)=2\cdot 0 
+0+0=0$ (see Definition \ref{def sfinv}).

For $n=k=1$, the left hand side is $e_1(1)=1$ and the right hand side is $ss^B_q(1,0)=1$, due to the unique signed permutation having a split cycle $(1)(-1)$, presented by the RG-word $\omega=\omega_1=(1,1)$, having 
${\rm sfinv}(\omega)=2\cdot 0 +0 +0=0$.  

Finally, for the case $k=n$, note that by substituting $k=0$ in Equation (\ref{eqn generating function of q analogue for the variation}), we have that:
$$e_n\left([1]_q,[3]_q,\dots, [2n-1]_q\right)=\prod_{i=0}^{n-1}[2i+1]_q \stackrel{\rm  (\ref{eqn generating function of q analogue for the variation})}{=}ss^B_q(n,0)=ss^B_q(n,n-n).$$

Now assume that the claim holds for $n-1 \geq 1$ and we prove it for $n$, where $k$ satisfies $0\leq k < n$:
\begin{eqnarray*}
 & \hspace{-305pt} e_k\left([1]_q,[3]_q,\dots ,[2n-1]_q\right)  \stackrel{(\ref{rec_ek})}{=} & \ \\
& \hspace{-62pt}= \hspace{12pt} e_k\left([1]_q,[3]_q,\dots ,[2n-3]_q\right)+[2n-1]_q e_{k-1}\left([1]_q,[3]_q,\dots ,[2n-3]_q\right)=&\\
& \hspace{-30pt}\stackrel{\rm induction}{=} 
ss^B_q(n-1,(n-1)-k)+[2n-1]_q ss^B_q(n-1,(n-1)-(k-1))
\stackrel{(\ref{recursion q shifted stirling B})}{=}
ss^B_q(n,n-k).
\end{eqnarray*}

\end{proof}

\begin{lemma}\label{specialization in h}
$$h_k \left([1]_q,[3]_q,\dots, [2n-1]_q\right)= S^B_q(n+k,n).$$
\end{lemma} 

\begin{proof}
We prove it by induction on $n+k$, where $0 \leq k \leq n$.
For $n+k=1$ we have to check only the case $n=1$ and $k=0$. 
The left hand side is $h_0(1)=1$, while the right hand side is $S^B_q( 1, 1)=1$, due to the unique set partition of type $B$ of the set $\{1\}$, having exactly one non-zero block:  $\{\{1\},\{-1\}\}$,
which contributes ${\rm wt}_1(\omega)=1$ (see Definition \ref{def q stirling second kind}).

Now assume that the claim holds for $n+k\geq 1$ and we prove it for $n+k+1$ (where $0\leq k \leq n$). 
We split our proof into two cases
according to the index we are increasing.
\medskip

\noindent
{\bf First case - Increment of the index $n$}:
\begin{eqnarray*}
 & \hspace{-280pt} h_k\left([1]_q,[3]_q,\dots ,[2n+1]_q\right)  \stackrel{(\ref{rec_hk})}{=} & \ \\
& \hspace{-23pt}= \hspace{12pt} h_k\left([1]_q,[3]_q,\dots ,[2n-1]_q\right)+[2n+1]_q h_{k-1}\left([1]_q,[3]_q,\dots ,[2n+1]_q\right)=&\\
& \hspace{-68pt}\stackrel{\rm induction}{=} 
S^B_q(n+k,n)+[2n+1]_q S^B_q(n+k,n+1)
\stackrel{(\ref{recursion q second kind})}{=}
S^B_q(n+k+1,n+1).
\end{eqnarray*}

\medskip

\noindent
{\bf Second case - Increment of the  index $k$:}
\begin{eqnarray*}
 & \hspace{-280pt} h_{k+1}\left([1]_q,[3]_q,\dots ,[2n-1]_q\right)  \stackrel{(\ref{rec_hk})}{=} & \ \\
& \hspace{-10pt}= \hspace{12pt} h_{k+1}\left([1]_q,[3]_q,\dots ,[2n-3]_q\right)+[2n-1]_q h_k\left([1]_q,[3]_q,\dots ,[2n-1]_q\right)=&\\
& \hspace{-75pt}\stackrel{\rm induction}{=} 
S^B_q(n+k,n-1)+[2n-1]_q S^B_q(n+k,n)
\stackrel{(\ref{recursion q second kind})}{=}
S^B_q(n+k+1,n).
\end{eqnarray*}

\end{proof}

\begin{remark}
Note that the case $q=1$ of Lemmas  \ref{specialization in e} and \ref{specialization in h} appeared as Theorem 2.1(a)-(b) in Sagan and Swanson \cite{SaSw}.
\end{remark}

We are now in position to prove Theorem \ref{orthogonality}:
\begin{proof}[Proof of Theorem \ref{orthogonality}]
The result follows by Lemmas \ref{specialization in e} and \ref{specialization in h}, specializing Equations (\ref{convolution}) and (\ref{weighted convolution}) at $[1]_q,[3]_q,\dots,[2n-1]_q$. \end{proof}

\section*{Acknowledgements}
We thank Bruce Sagan and Josh Swanson for many fruitful discussions.


\begin{thebibliography}{99}

\bibitem{AR} R.M. Adin and Y. Roichman,
{\it The flag major index and group actions on polynomial rings},
European J. Combin. {\bf 22}(4)
(2001), 431--446.

\bibitem{monthly} 
E. Bagno and D. Garber, {\em 
Signed partitions - A `balls into urns' approach}, 
Bull. Math. Soc. Sci. Math. Roumanie {\bf 65 (113)}, No. 1, 63--71 (2022).

\bibitem{PUMA}
E. Bagno, D. Garber and
T. Komatsu, {\it A $q, r$-analogue for the Stirling numbers of the second kind of Coxeter groups of type $B$},
Pure Math. Appl. {\bf 30}(1) (2022), 8--16.

\bibitem{Bala} 
P. Bala, {\it A 3-parameter family of generalized Stirling numbers} (2015). Electronic version: https://oeis.org/A143395/a143395.pdf

\bibitem{BQ} A.T. Benjamin and J.J. Quinn, {\it Proofs That Really Count: The Art of Combinatorial Proof}, 1st ed., vol. {\bf 27}, Mathematical Association of America, 2003. 

\bibitem{Beno} 
M. Benoumhani, {\em 
On Whitney numbers of Dowling lattices}, 
Discrete Math. {\bf 159} (1996), 13--33.

\bibitem{Broder}  
A. Z. Broder, {\em  
The $r$-Stirling numbers},  
Discrete Math. {\bf 49} (1984), 241--259. 

\bibitem{CaRe}
Y. Cai and M. A. Readdy,
{\em $q$-Stirling numbers: A new view}, 
Adv. Appl. Math. {\bf 86} (2017), 50--80.


\bibitem{d'Ocagne} M. d'Ocagne, {\it Sur une classe de nombres remarquables}, Amer. J. Math. {\bf  9} (1887), 353--380. Online at: https://archive.org/details/jstor-2369478

\bibitem{DoLu} I. Dolgachev and V. Lunts, {\it A character formula for the representation of a Weyl group in the cohomology of the associated toric variety}, J. Algebra {\bf 168}(3) (1994),  741--772. 

\bibitem{Dow} 
T. A. Dowling, {\em 
A class of geometric lattices based on finite groups}, 
J. Combin. Theory, Ser. B {\bf 14} (1973), 61--86. Erratum: J. Combin. Theory, Ser. B {\bf 15} (1973), 211.

\bibitem{EgGa} \"O. E\u{g}ecio\u{g}lu and A.M. Garsia, {\it  Lessons in Enumerative Combinatorics}, Graduate Texts in Math. Vol. {\bf 290}, Springer, 2021. 

\bibitem{Egge} E.S. Egge, {\it An introduction to symmetric functions and their combinatorics}, Student Math. Library,  Vol. {\bf 91}, Amer. Math. Soc., 2019.


\bibitem{GKP} R.L. Graham, D.E. Knuth and O. Patashnik, {\it Concrete Mathematics: A Foundation for Computer Science}, second edition, Addison-Wesley, Reading, Mass., 1994.

\bibitem{GZ05} V. J. W. Guo and J. Zeng, {\em   A $q$-analogue of Faulhaber's formula for sums of powers},  Electron. J. Combin. {\bf 11} (2004/2006), \#R19, 19 pp. 


\bibitem{Hut} G. Hutchinson, {\em Partitioning algorithms for finite sets},  Commun. ACM {\bf 6} (1963), 613--614.

\bibitem{Knuth93} D. E. Knuth, {\it Johann Faulhaber and sums of power}, Math. Comp. {\bf 61} (1993), 277--294.


\bibitem{merca15}   
M.  Merca,   {\it An alternative to Faulhaber’s formula}, Amer. Math. Monthly {\bf 122}(6) (2015), 599--601. 

\bibitem{merca18}   
M. Merca, {\em  
A $q$-analogue for sums of powers}, 
Acta Arith. {\bf 183}(2) (2018), 185--189. 


\bibitem{Mezo} 
I. Mez\H{o}, {\em 
Combinatorics and Number Theory of Counting Sequences}, CRC Press. Taylor \& Francis
Group. Boca Raton (FL), 2020.


\bibitem{Milne} 
S. Milne,
{\em 
Restricted growth functions and incidence relations of the lattice of partitions of an $n$-set}, 
Adv. Math. {\bf 26} (1977), 290--305. 

\bibitem{OEIS} OEIS Foundation Inc. (2022), The On-Line Encyclopedia of Integer Sequences, http://oeis.org

\bibitem{R} 
V. Reiner, {\em 
Non-crossing partitions for classical reflection groups}, 
Discrete Math. {\bf 177}(1--3) (1997), 195--222.

\bibitem{Sagan} 
B. Sagan, private communication, 30.1.2022.

\bibitem{Sagan personal} 
B. Sagan, private communication, 24.1.2022.

\bibitem{SaganBook} 
B. Sagan, 
{\it Combinatorics: the art of counting}, Graduate Studies in Math. {\bf 210}, Amer. Math. Soc., 2020.

\bibitem{SaSw} 
B. Sagan and J. Swanson, {\em 
$q$-Stirling numbers in type $B$}, European J. Combin. {\bf 118} (2024), 103899. 


\bibitem{EC1} R.P. Stanley, {\it Enumerative Combinatorics}, Vol. 1, Second edition, Cambridge University Press, 2012.

\bibitem{EC2} R.P. Stanley, {\it Enumerative Combinatorics}, Vol. 2, Cambridge University Press, 1999.



\bibitem{Wa} 
D. G. L. Wang, {\em 
On colored set partitions of type $B_n$}, 
Cent. Eur. J. Math. {\bf 12}(9) (2014), 1372--1381.


\bibitem{Za} T. Zaslavsky, {\it The geometry of root systems and signed graphs}, Amer. Math. Monthly {\bf 88}(2) (1981),
88--105.

\end{thebibliography}
\end{document}